\documentclass{siamart220329}
\usepackage[utf8]{inputenc}

\usepackage{appendix}
\usepackage{amsmath}
\usepackage{graphicx}
\usepackage{stmaryrd}
\usepackage[margin=1in]{geometry}

\usepackage{amsfonts}

\usepackage{color}
\usepackage{xcolor}
\usepackage{braket}
\usepackage{url}
\usepackage{hyperref}
\hypersetup{
    colorlinks,
    linkcolor={red!50!black},
    citecolor={blue!50!black},
    urlcolor={blue!20!black}
}

\newcommand{\order}{{\mathcal O}}
\newcommand{\avg}[1]{\braket{#1}}
\newcommand{\hq}{\overline{q}}
\newcommand{\heta}{\overline{\eta}}
\newcommand{\mean}[1]{\braket{#1}}
\newcommand{\fluctint}[1]{\llbracket #1 \rrbracket}
\newcommand\Hm[1]{\mean{H^{#1}}}
\newcommand\braced[1]{\left\{{#1}\right\}}

\newcommand{\muhat}{\hat{\mu}}
\newcommand{\FF}{\mathcal{F}}
\newcommand{\bhat}{\hat{\beta}}

\newtheorem{lem}{Lemma}
\newtheorem{prop}{Proposition}
\newtheorem{cor}{Corollary}
\newtheorem{rem}{Remark}

\title{A multiscale model \\
for weakly nonlinear shallow water waves \\
over periodic bathymetry}
\author{David I. Ketcheson\thanks{{\texttt{david.ketcheson@kaust.edu.sa}}, Applied Mathematics and Computational Science, CEMSE Division, King Abdullah University of Science and Technology (KAUST), Thuwal, 23955-6900, Kingdom of Saudi Arabia}
\and Lajos L\'oczi\thanks{{\texttt{LLoczi@inf.elte.hu}}, Department of Numerical Analysis, ELTE E\"otv\"os Lor\'and University, P\'azm\'any P.~s.~1/C, H-1117 Budapest, Hungary, and Department of Analysis and Operations Research, BME Budapest University of Technology and Economics}
\and Giovanni Russo\thanks{\texttt{giovanni.russo1@unict.it}, Department of Mathematics and Computer Science, University of Catania, Viale A.\ Doria 6, 95125 Catania, Italy}}

\begin{document}

\maketitle

\begin{abstract}
We study the behavior of shallow water waves over periodically-varying bathymetry, based on the first-order hyperbolic
Saint-Venant equations.  Although solutions of this system are known to generally exhibit wave breaking, numerical experiments suggest
a different behavior in the presence of periodic bathymetry.
Starting from the first-order
variable-coefficient hyperbolic system, we apply a multiple-scale perturbation approach
in order to derive a system of constant-coefficient 
high-order partial differential equations whose solution approximates
that of the original system.  The high-order system turns out to
be dispersive and exhibits solitary-wave formation, in close agreement with direct numerical simulations of the original system.
We show that the constant-coefficient homogenized system 
can be used to study the properties of solitary waves and to
conduct efficient numerical simulations.
\end{abstract}

\section{Introduction}
Linear wave propagation in periodic media has been extensively studied in a variety of contexts.
In solid state physics, the behavior of optical and acoustic waves in crystals has been understood for more than a century, including associated effects like Bragg diffraction and band-gaps, which result from the interaction of the wave with the periodic structure (see, for example, \cite{ashcroft2022solid}).
Such effects become significant when the period of the structure is small but not negligible with respect to the typical length scale of the wave. 

More recently, periodic materials have been engineered specifically to obtain desired wave propagation properties. 
Metamaterials are composite materials obtained by assembling a large number of small unit cells, such that on a length scale much greater than the cellular size, they behave as a homogeneous material with different properties from the original constituent materials. 
In the most interesting cases they exhibit properties that are not found in natural homogeneous materials and are not even intermediate between the properties (mechanical or optical) of the constituents.
An enormous literature on metamaterials is available, and is mostly focused on
electromagnetic, acoustic, or elastic waves.  Recently, metamaterials (including
effects like cloaking) have been studied in the context of water waves \cite{berraquero2013experimental,porter2017cloaking}.  In this context a periodic
medium is introduced by periodic variation of the bottom elevation (bathymetry).
Analysis of the linear propagation of water waves over periodic bathymetry is facilitated
through perturbation theory; see for instance the recent works \cite{maurel2019scattering,porter2019extended,marangos2021shallow} for
specific application to water waves and earlier works such as \cite{santosa1991,quezada_dispersion} for more general analysis of linear waves.

The nonlinear behavior of waves in periodic structures has received
relatively less attention.  Indeed, the
analysis of the properties of acoustic metamaterials is in general based on linear or linearized equations, so that it is assumed that the displacement, or any type of signal, is sufficiently small so that linear theory can be used, or that nonlinear effects can be computed by perturbation methods, still assuming sufficiently small signals. 
However, in real-world settings there are often situations in which
nonlinear effects are crucial and cannot be neglected. 


 Perturbation techniques are also commonly used to study weakly nonlinear effects, thus capturing essential features of the phenomenon, without the complexity of the full nonlinear system. 
 In \cite{leveque2003}, propagation of waves in a layered elastic medium has been studied. The 
 medium is formed by a large number of alternating layers of two materials,
 each one with its unperturbed density and strain-stress relation. All layers of the same material have the same unperturbed thickness. 
Under the assumption that the period of the unperturbed multimedia (i.e., the thickness of a double layer) is considerably smaller than the wavelength, 
detailed numerical computation shows 
the appearance of wave patterns typical of dispersive waves. 
Using perturbation methods, the authors were able to derive effective equations for the propagation of waves in a homogenized medium, which are in good agreement with the detailed numerical simulation of the wave propagation in the multilayer system. The agreement improves if more terms are included in the perturbation expansion. 

Multilayered fluids have also been considered in the literature. 
In \cite{phan2023numerical}, for example, the behaviour of a large number of pairs of layers is studied numerically. Each pair is formed by two different fluids, each one treated as a stiffened gas with its own standard density $\rho$, adiabatic exponent $\gamma$, and a constant $\Pi$ which determines the stiffness of the fluid, and which is sometimes called attractive pressure \cite{Chiapolino_Saurel_18}. In the paper it is shown that a simple isentropic homogenized model is able to capture the behavior of the solution to the multilayer problem, but only before shocks develop.  

Water waves present an interesting scenario for the study of nonlinear waves in a periodic medium.  Here the periodicity is introduced through variation in the bathymetry, or bottom elevation.
This is the subject of the present work.  In order to understand possible dispersive effects introduced by the bathymetry, separate from other dispersive effects present in water waves over a flat bottom, we start from the non-dispersive shallow water equations.
We apply a perturbation technique similar to that employed in \cite{leveque2003}, with the goal of deriving 
a set of homogenized effective equations that approximate to high accuracy the solution of the variable-bathymetry shallow water equations.

In addition to the literature on water wave metamaterials mentioned already,
many other previous works have examined the effect of bathymetry on water waves.
Starting from the equations for irrotational, inviscid flow with a free surface,
it has been shown that solitary waves in shallow water over periodic bathymetry obey a KdV-type equation
with modified velocity and dispersion \cite{rosales1983gravity}.
Effective dispersion of shallow water waves over periodic bathymetry was also
studied in \cite{2021_solitary}, where the configuration of the bathymetry to
the waves is rotated by 90 degrees relative to what is studied in the present
work.
Here we focus on waves that are (somewhat) long relative to the
bathymetric variation;
in \cite{Benilov2006}, the authors study the opposite relative scaling, in which the waves
oscillate much more rapidly than the bathymetry.
Multiple-scale analysis was also applied to the shallow water equations in \cite{yong2002initial}, with flat bathymetry and a rapidly-oscillating boundary.

The plan of the paper is the following. The next subsection is devoted to the description of the problem set-up. 
In Section \ref{sec:homog} we perform the multiple-scale analysis which leads to the effective equations to various order in the small parameter. In Section \ref{sec:numerical} we perform comparisons with a detailed numerical solution of the Saint-Venant equations over periodic bathymetry. In the last section we draw conclusions and discuss directions for future work. 

The analysis in Section \ref{sec:homog} is facilitated by a certain averaging operator $\fluctint{\cdot}$, defined therein.
In the Appendix, we state and prove several properties of this operator that are essential
in simplifying the analysis.


\subsection{Model equations and assumptions}
In this work we study the shallow water wave (or Saint-Venant) model:
\begin{subequations} \label{eq:sw1}
\begin{align}
    h_t + (hu)_x & = 0 \\
    (hu)_t + \left( hu^2 + \frac{1}{2}gh^2 \right)_x & = - g h b_x,
\end{align}
\end{subequations}
where $h(x,t)$ and $u(x,t)$ denote, respectively the water depth and the depth-averaged velocity, $b(x)$ denotes the bottom elevation (bathymetry), and subscripts denote partial derivatives with respect to the corresponding variables.
We are interested in the behavior of waves propagating over periodic bathymetry with period $\delta$:
$$
    b(x+\delta) = b(x).
$$
We focus on waves whose wavelength is long relative to $\delta$, and scenarios
in which the variation in $b(x)$ is of the same order as the overall depth (but
not so large that dry states appear).  The notation and scales involved are
depicted in Figure \ref{fig:scenario}.  To facilitate the analysis that will follow,
in this work we assume that $b(x)$ is continuously differentiable.
However, numerical experiments suggest that the regularity of $b$ has
little effect on the qualitative behavior discussed herein.
\begin{figure}
    \includegraphics[width=\textwidth]{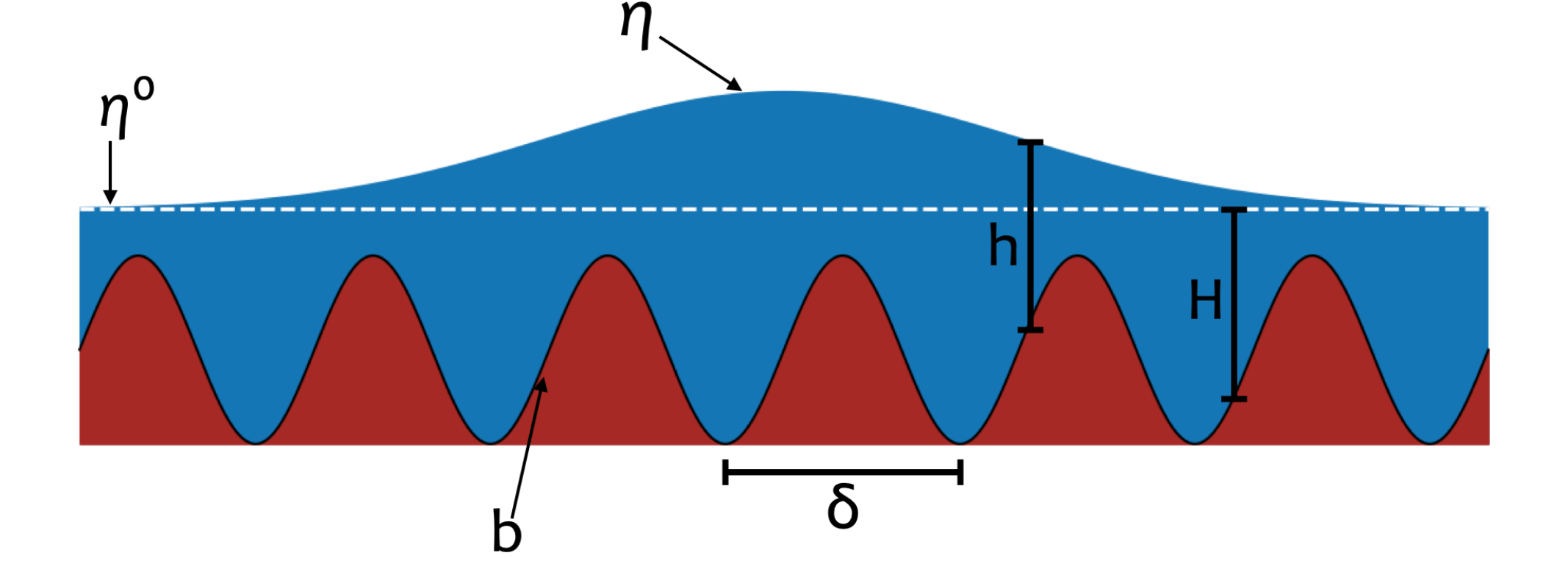}
    \caption{Depiction of notational conventions used in this work\label{fig:scenario}}
\end{figure}
Throughout the paper we use SI units for all physical quantities.
In Figure \ref{fig:example1} we show an example of the surprising behavior exhibited by
waves under these conditions.  Here we take $\delta=1$, $\eta^0=0$, and
$$
    b(x) = \begin{cases} -1 & 0 \le x - \lfloor x \rfloor < 1/2,  \\
    -3/10 & 1/2 \le x - \lfloor x \rfloor < 1.
    \end{cases}
$$
The initial velocity $u$ is zero, while the initial water depth $h$ is a Gaussian pulse of amplitude $1.5\times 10^{-2}$, which splits into
two symmetric pulses.  The figures show the evolution of the right-going pulse.
Whereas solutions of \eqref{eq:sw1} (like those
of other first-order hyperbolic systems) generically exhibit wave breaking
after a short time, in this example the initial pulse breaks up into a train of
what are apparently traveling waves.  These appear to be globally attractive and
stable solitary waves, similar to those observed in other hyperbolic systems with
periodic coefficients \cite{leveque2003,2014_cylindrical_solitary_waves,2021_solitary}.
For comparison, we plot in blue the solution obtained with a flat bottom,
which shows the expected behavior of $N$-wave formation and subsequent decay.

\begin{figure}
    \includegraphics[width=\textwidth]{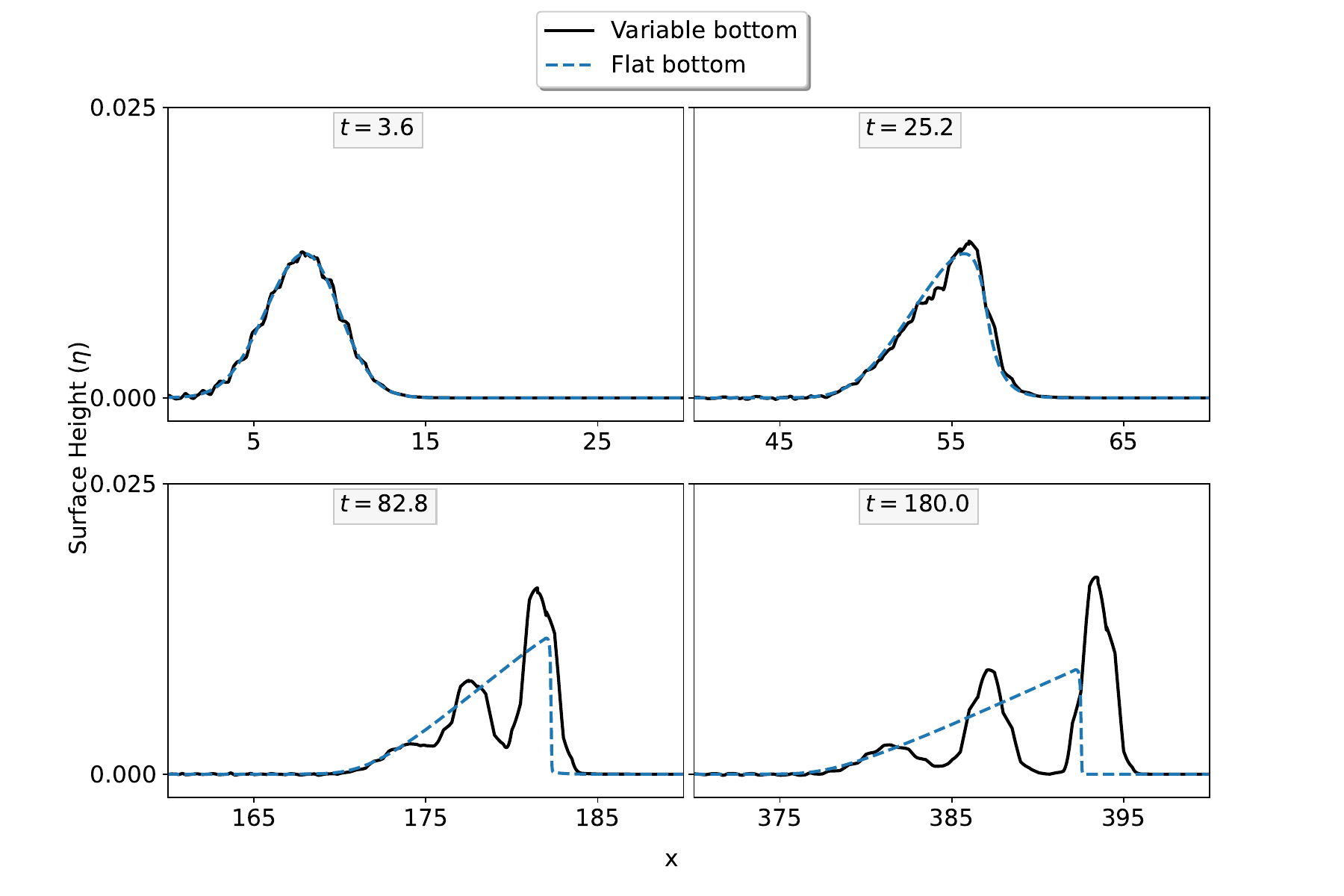}
    \caption{Evolution of an initial Gaussian pulse over periodic bathymetry.
        The surface elevation is shown, measured in meters. 
 For comparison, the dashed blue line shows the solution for flow over a flat bottom.\label{fig:example1}}
\end{figure}

In this work we seek to explain and further understand this behavior.  We
employ a multiple-scale analysis following a technique pioneered by Yong and
coauthors \cite{yong2002,leveque2003}.  This leads to a constant-coefficient
wave equation that includes a dispersive term, indicating that the periodic
bathymetry induces an effective dispersion of long waves, similar to what has
been found for other hyperbolic systems with periodically-varying coefficients
\cite{leveque2003}.  After deriving this \emph{homogenized} equation in Section
\ref{sec:homog}, we investigate its properties in Section \ref{sec:analysis} and compare
numerical solutions of the first-order and homogenized equations in Section
\ref{sec:numerical}.  We also investigate the properties and dynamics of the
solitary wave solutions appearing in this system. 

Code to reproduce all time-dependent simulation results in this paper can be found online\footnote{\url{https://github.com/ketch/shallow_water_bathymetry_effective_medium_RR}}.

\subsection{Contributions}
The main new contribution of the present work is the application of multiple-scale homogenization to the shallow water equations, demonstrating that shallow water waves over periodic bathymetry exhibit effective dispersion.  Additional contributions include:
\begin{itemize}
    \item analysis of various forms of the homogenized equations and their stability;
    \item analysis and numerics of traveling wave solutions;
    \item demonstration of how fine-scale structure
    can be recovered using the homogenized equations;
    \item discovery and proof of several new properties of the averaging operators.
\end{itemize}
As mentioned already, we have made all of our code available and we consider this also to be a significant contribution.  In particular, the \textit{Mathematica} notebooks accompanying this work provide a substantial systematization of the homogenization process and will drastically simplify its application to other systems in the future.

The application of Yong--Kevorkian homogenization \cite{yong2002, yong2002initial} to the shallow water equations is, to some degree, mechanical.  However, it can be carried out only after finding an appropriate form of the equations from which to start; this choice may seem obvious in hindsight but was far from trivial.  Furthermore, the \emph{mechanical} homogenization process is nevertheless extremely challenging to carry out---both for the human and the computer involved.  It would be impossible without the discovery and application of the new averaging operator properties described in the Appendix; see Table \ref{tbl:terms} and the accompanying discussion.

\section{Multiple-scale analysis}\label{sec:homog}
We begin by rewriting the shallow water system \eqref{eq:sw1} in terms of 
the surface elevation $\eta = h+b$ and the discharge $q=hu$:
\begin{subequations} \label{eq:eta-q}
\begin{align}
    \eta_t + q_x & = 0 \label{eq:2a} \\
    q_t + \left( \frac{q^2}{\eta-b}\right)_x + g(\eta-b)\eta_x & = 0. \label{eq:2b}
\end{align}
\end{subequations}
We will exploit a separation of spatial scales, and
the advantage of working with $q$ instead of $u$ is that the velocity
$u$ tends to vary as rapidly (in space) as $b$, whereas the discharge $q$ can be
slowly-varying even when $b$ is rapidly-varying.  
In the derivation of the equations we shall act formally, assuming sufficient regularity of the functions. 
Furthermore, the solutions in which we are most interested appear not to contain shocks, as we will see.  For this reason it is not essential to work with a conservative version of the equations.

We introduce a small parameter $\delta$ and fast spatial scale $y=x/\delta$.  
The two spatial scales are treated formally as independent variables, so that 
\begin{align} \label{eq:dtrans}
\frac{\partial}{\partial x} \to \frac{\partial}{\partial x} + \delta^{-1} \frac{\partial}{\partial y}.
\end{align}
We assume there exists a power series expansion for each in terms of $\delta$:
\begin{subequations} \label{delta-series}
\begin{align}
    \eta & = \eta^0(x,t) + \delta \eta^1(x,t,y) + \delta^2 \eta^2(x,t,y) + \cdots \\
    q & = q^0(x,t) + \delta q^1(x,t,y) + \delta^2 q^2(x,t,y) + \cdots
\end{align}
\end{subequations}
Here superscripts on $\eta$ and $q$ are indices, while superscripts on other quantities are exponents.  We follow this convention through the remainder of Section \ref{sec:homog}.
When powers of these quantities are needed, we will use parentheses; e.g., $(q^1)^2$.
All functions appearing in \eqref{delta-series} 
are assumed to be periodic in $y$, with period 1, 
so that the integral of any
$y$-derivative over one period must vanish.

As we will see below, it turns out that $\eta^0$ is constant.
It is therefore convenient to introduce the function $$H(y) := \eta^0-b(y).$$
We assume throughout the paper that $H(y)>0$
and note that $b'(y)=-H'(y)$;
see Figure \ref{fig:scenario}.

Next we make the above substitutions in \eqref{eq:eta-q} and collect terms for each power of $\delta$.
We immediately see that there is only one term proportional
to $\delta^{-1}$, coming from equation \eqref{eq:2b}:
$$
    \delta^{-1} \frac{(q^0)^2 H'(y)}{(\eta-b)^2}.
$$
This implies that $q^0=0$, i.e.,
$q=\order(\delta)$.
Taking this into account, the expansion of \eqref{eq:eta-q} takes the form
\begin{subequations} \label{initial-expansion}
\begin{align}
    \eta^0_t + q^1_y + \delta \left(q^2_y + q^1_x + \eta^1_t\right)
          + \delta^2 \left(q^3_y + q^2_x + \eta^2_t \right) & =  \order(\delta^3) \\
        g \left(\eta^0_x + \eta^1_y\right)H + \delta \Big( q^1_t + 2q^1 q^1_y H^{-1} + g((\eta^0_x + \eta^1_y)\eta^1 + (\eta^1_x +\eta^2_y)H) - (q^1)^2 H' H^{-2} \Big) \nonumber \\
        + \delta^2 \Big(2q^1 q^1_x H^{-2}+ q^2_t + 2((q^2 - q^1 \eta^1 H^{-1})q^1_y + q^1 q^2_y)H^{-1}  - (q^1)^2 \eta^1_y H^{-2}  \nonumber \\
        + g( (\eta^0_x + \eta^1_y)\eta^2 + (\eta^1_x+\eta^2_y)\eta^1+(\eta^2_x + \eta^3_y)H)-2 (q^2 - q^1 \eta^1 H^{-1})q^1 H'H^{-2} \Big) & = \order(\delta^3).
\end{align}
\end{subequations}
Here to save space we have displayed only the terms up to $\order(\delta^2)$; in what follows we will make use
of these expansions to even higher order.  As we can see, the first equation
has a simple structure, but the second equation has a
rapidly-increasing number of terms at each order.
In what follows, we repeat the following steps
(following a process developed and applied earlier in
\cite{yong2002,yong2002initial,leveque2003}):
\begin{enumerate}
    \item Equate terms of the same order in $\delta$ and solve for the terms with highest index;
    \item Average the resulting equations with respect to $y$;
    \item Integrate to find formulas for the highest-index variables in terms of $y$-averages of lower-index variables and the function $H(y)$.
\end{enumerate}
Eventually we will obtain equations for the $y$-averages of $\eta$ and $q$, by summing the result obtained at each order in step 2.  Step 3 is required in order to allow us to determine the $y$-dependence explicitly at subsequent orders.

\subsection{Averaging operators}\label{sec:averaging}
The following operators will appear frequently in our analysis.  First, the integral (or average) of $f$ over one period, denoted by $\mean{f}\in\mathbb{R}$, is defined as
\[
    \mean{f} := \int_0^1 f(y)\, dy. 
\]
Second, the fluctuating part of the function $f$, denoted by $\{ f \}$, is defined as 
$$\{f\}(y):=f(y)-\mean{f}.$$
Finally, the fluctuating part of the antiderivative of the fluctuating part, denoted by 
$\llbracket f \rrbracket$, is defined for any $y$ as
\[
\fluctint{f}(y):= \left\{\int_0^y \braced{f(\xi)}\, d\xi\right\},
\]
that is,
\begin{equation}\label{[[]]expldef}
    \fluctint{f}(y) = \int_0^y\left\{f\right\}(\xi)d\xi  - \int_0^1 \int_0^\tau \left\{f\right\}(\xi)d\xi d\tau.
\end{equation}
Clearly, we have 
\[
\mean{\{f\}}=0 \quad \quad \text{and} \quad \quad \mean{\fluctint{f}}=0.
\]

Some useful properties of the $\fluctint{\cdot}$ operator, 
which will be used throughout the paper, are provided in Appendix \ref{sec:appendix}.

\begin{rem}
    We will often write $\braket{f}$
even for functions $f$ that are independent of $y$ \emph{a priori}, in order to emphasize which factors do not depend on $y$. 
Also, note that while $\mean{f}$ is $y$-independent, $\{f\}$ and $\fluctint{f}$ depend on $y$.
\end{rem}

\subsection{$\order(\delta^0)$}
We now follow the 3 steps outlined above at each order, starting with terms proportional to $\delta^0$.
Equating these terms and solving for the highest-index terms (step 1) yields
\begin{subequations} \label{order0}
\begin{align}
   -q^1_y & = \eta^0_t \\
   -\eta^1_y & = \eta^0_x.
\end{align}
\end{subequations}
Next (step 2) we integrate the equations above over one
period with respect to $y$.  The left-hand-sides vanish,
since they represent the $y$-integral (over one period)
of the $y$-derivative of a periodic function.
Thus we have $\eta^0_t = \eta^0_x = 0$, so $\eta^0$ is
a constant.

In this case we can skip step 3, since we see immediately that
\begin{align*}
    q^1(x,t,y) & = \avg{q^1(x,t)} \\
    \eta^1(x,t,y) & = \avg{\eta^1(x,t)}.
\end{align*}
Thus the expansion \eqref{delta-series} simplifies to
\begin{subequations}
\begin{align}
    \eta & = \eta^0 + \delta \eta^1(x,t) + \delta^2 \eta^2(x,t,y) + \cdots \\
    q & = \delta q^1(x,t) + \delta^2 q^2(x,t,y) + \cdots
\end{align}
\end{subequations}
We see that the waves in which we are interested occur
as perturbations to a flat surface, still water state.

\subsection{$\order(\delta^1)$}
We proceed to follow the same steps, focusing on the next-order terms.
Collecting terms proportional to $\delta^1$ and solving for the highest-index terms (step 1) gives
\begin{subequations} \label{order1}
\begin{align}
    -q^2_y & = \eta^1_t + {q^1_x} \label{dq2dy} \\
    -\eta^2_y & = {\eta^1_x} + \frac{{q^1_t}}{gH} - \frac{(q^1)^2 H'}{gH^3}. \label{deta2dy}
\end{align}
\end{subequations}
Note that the additional terms appearing in \eqref{initial-expansion} vanish because 
$\eta^0_x=0$ and $q^1$ and $\eta^1$ do not depend on $y$.
 Next (step 2) we integrate both sides of the above equations over one period in $y$.  The integral of the left hand side vanishes since all functions are periodic in $y$:
\begin{subequations} \label{order1mean}
\begin{align}
    \avg{\eta^1_t} + \avg{q^1_x} & = 0 \label{order1a} \\ 
    \avg{\eta^1_x} + \frac{\mean{H^{-1}}}{g} \avg{q^1_t} - \frac{\mean{H'H^{-3}}}{g} \avg{q^1}^2 & = 0. \label{order1b} 
\end{align}
\end{subequations}
Note that  $\eta^1 = \avg{\eta^1}$, $q^1 = \avg{q^1}$ and $\avg{\cdot}$ commutes with space and time derivatives,
so $\eta^1_t = \avg{\eta^1_t}$, $q^1_x = \avg{q^1_x}$,
and that 
Eq.~\eqref{order1b} can be simplified since $\mean{H' H^{-3}}=0$ 
as it is the average of the derivative of a periodic function (see property \eqref{FTC} in the Appendix).  
So system \eqref{order1mean} simplifies to
\begin{subequations} \label{order1meansimp}
\begin{align}
    \avg{\eta^1_t} + \avg{q^1_x} & = 0 \label{order1meansimp-a} \\ 
    \avg{q^1_t} + \frac{g}{\mean{H^{-1}}} \avg{\eta^1_x} & = 0. \label{order1meansimp-b}
\end{align}
\end{subequations}
This is simply the linear wave equation, 
with sound speed
$$
    c = \sqrt{g/\mean{H^{-1}}},
$$
which is an averaged version of the usual characteristic speed $\sqrt{gh}$ for gravity waves
in shallow water.

Finally we carry out step 3.  Subtracting \eqref{order1} from \eqref{order1mean} yields 
\begin{align*}
    q^2_y & = 0 \\
    \eta^2_y & = -\frac{\braced{H^{-1}}}{g} \avg{q^1_t} + \frac{\braced{H' H^{-3}}}{g}\avg{q^1}^2.
\end{align*}
Next we integrate with respect to $y$, obtaining
\begin{align*}
    \int_0^y q^2_s(x,t,s)ds & = q^2(x,t,y) - q^2(x,t,0) = 0 \implies q^2 = \avg{q^2}
\end{align*}
and
\[
    \int_0^y \eta^2_s(x,t,s)ds  = \int_0^y \left(-\frac{\braced{H(s)^{-1}}}{g} \avg{q^1_t} + \frac{\braced{H'(s) H(s)^{-3}}}{g}\avg{q^1}^2\right) ds;
\]
therefore 
\[
    \eta^2(x,t,y) = \eta^2(x,t,0) + \int_0^y \left(-\frac{\braced{H(s)^{-1}}}{g} \avg{q^1_t} + \frac{\braced{H'(s) H(s)^{-3}}}{g}\avg{q^1}^2\right) ds. 
\]
Subtracting the mean we obtain 
\[
    \eta^2  = \avg{\eta^2} - \frac{\fluctint{H^{-1}}}{g} \avg{q^1_t} + \frac{\fluctint{H'H^{-3}}}{g} \avg{q^1}^2.
\]
In summary we have obtained (also simplifying the last term a bit more)
\begin{subequations} \label{order2pre}
\begin{align}
q^2 & = \avg{q^2} \\
\eta^2 & = \avg{\eta^2} - \frac{\fluctint{H^{-1}}}{g} \avg{q^1_t} + \frac{\fluctint{H'H^{-3}}}{g} \avg{q^1}^2 = \avg{\eta^2} - \frac{\fluctint{H^{-1}}}{g} \avg{q^1_t} - \frac{\braced{H^{-2}}}{2g}\avg{q^1}^2 \label{eta2}.
\end{align}
\end{subequations}
From \eqref{order2pre} we see that $\eta^2$ depends on the fast scale $y$, while $q^2$ does not.

\subsection{$\order(\delta^2)$}
The $\order(\delta^2)$ equations are
\begin{subequations} \label{order2}
\begin{align}
    -q^3_y & = \eta^2_t + \avg{q^2_x} \\
    -\eta^3_y & = \eta^2_x + H^{-1} \avg{\eta^1}(\avg{\eta^1_x} + \eta^2_y) \\ \nonumber
              & + g^{-1}\left( 2H' H^{-4} \avg{\eta^1} \  \avg{q^1}^2 +H^{-1} \avg{q^2_t} - 2H'H^{-3}\avg{q^1} \ \avg{q^2} + 2H^{-2} \avg{q^1} \ \avg{q^1_x}    \right).
\end{align}
\end{subequations}
Next we want to write the right-hand side of \eqref{order2} in terms of only $H(y)$ and quantities that are independent of $y$.  Therefore we use \eqref{eta2} and \eqref{deta2dy} to replace $\eta^2$ and $\eta^2_y$, obtaining
\begin{subequations} \label{order2subs}
\begin{align}
    -q^3_y & = \avg{\eta^2_t} - \frac{\fluctint{H^{-1}}}{g} \avg{q^1_{tt}} - \frac{\braced{H^{-2}}}{2g} (\avg{q^1}^2)_t + \avg{q^2_x} \\
    -g\eta^3_y & = g\avg{\eta^2_x} - \fluctint{H^{-1}} \avg{q^1_{tx}} - \frac{\braced{H^{-2}}}{2}(\avg{q^1}^2)_x
    - H^{-2} \avg{\eta^1} \left(\avg{q^1_t} - \frac{\avg{q^1}^2 H'}{H^2}\right) \nonumber \\
        & \ \ \ + 2H' H^{-4} \avg{\eta^1} \ \avg{q^1}^2 +H^{-1} \avg{q^2_t} - 2H'H^{-3}\avg{q^1} \ \avg{q^2} + 2H^{-2} \avg{q^1} \ \avg{q^1_x} \label{12b} \\
    & = g\avg{\eta^2_x} - \fluctint{H^{-1}} \avg{q^1_{tx}} - \frac{\braced{H^{-2}}}{2} (\avg{q^1}^2)_x
    - \frac{\avg{\eta^1}\avg{q^1_t}}{H^2} + 3\frac{\avg{\eta^1} \ \avg{q^1}^2 H'}{H^4} \nonumber \\
        & \ \ \ + H^{-1} \avg{q^2_t} - 2H'H^{-3}\avg{q^1} \ \avg{q^2} + 2H^{-2} \avg{q^1} \ \avg{q^1_x}. \label{12c} 
\end{align}
\end{subequations}

Next we integrate these equations over one period with respect to $y$. Several terms vanish due to the facts that $\mean{\fluctint{f}}=0$ by definition and property \eqref{FTC}. 
The equations simplify to
\begin{subequations} \label{order2mean}
\begin{align}
\avg{\eta^2_t} + \avg{q^2_x} & = 0 \\
\avg{q^2_t} + \frac{g}{\mean{H^{-1}}} \avg{\eta^2_x} - \frac{\mean{H^{-2}}}{\mean{H^{-1}}} \ \avg{\eta^1} \ \avg{q^1_t} +  2 \frac{\mean{H^{-2}}}{\mean{H^{-1}}}  \avg{q^1} \ \avg{q^1_x} & = 0. \label{order2mean-mom}
\end{align}
\end{subequations}
To obtain \eqref{order2mean-mom} we have additionally divided by $\mean{H^{-1}}$, as this form will be convenient later.
Subtracting \eqref{order2subs} from \eqref{order2mean} (but without dividing by $\mean{H^{-1}}$ in \eqref{order2mean-mom}) and integrating from zero to $y$ gives
\begin{subequations} \label{q3eta3}
\begin{align}
q^3(x,t,y) & = \avg{q^3(x,t)} + \frac{\fluctint{\fluctint{H^{-1}}}}{g} \avg{q^1_{tt}} + \frac{\fluctint{H^{-2}}}{2g} (\avg{q^1}^2)_t \\
g\eta^3(x,t,y) & = g\avg{\eta^3(x,t)} + \fluctint{\fluctint{H^{-1}}}\avg{q^1_{tx}} + \frac{1}{2} \fluctint{H^{-2}} (\avg{q^1}^2)_x
    + \fluctint{H^{-2}}\avg{\eta^1} \ \avg{q^1_t} + \braced{H^{-3}}\avg{\eta^1} \ \avg{q^1}^2 \nonumber \\
    & \ \ - \fluctint{H^{-1}}\avg{q^2_t} - \braced{H^{-2}}\avg{q^1} \ \avg{q^2} - 2\fluctint{H^{-2}} \avg{q^1} \ \avg{q^1_x}.
    \label{eta3}
\end{align}
\end{subequations}

\subsection{$\order(\delta^2)$ governing equations}
Adding $\delta$ times \eqref{order1meansimp} with $\delta^2$ times \eqref{order2mean}, we obtain
the approximations
\begin{subequations} \label{o2avg}
\begin{align}
    \delta (\avg{\eta^1} + \delta \avg{\eta^2})_t + \delta (\avg{q^1} + \delta\avg{q^2})_x & = \order(\delta^3) \\
    \frac{\mean{H^{-1}}}{g}\delta (\avg{q^1} + \delta \avg{q^2})_t + \delta (\avg{\eta^1} + \delta \avg{\eta^2})_x
    - \delta^2 \frac{\mean{H^{-2}}}{g}\left(\avg{\eta^1} \avg{q^1_t} - (\avg{q^1}^2)_x\right)  & =\order(\delta^3).
\end{align}
\end{subequations}
Defining the averaged variables
\begin{subequations}
\begin{align}
    \heta & = \avg{\eta^1} + \delta \avg{\eta^2} + \dots \\
    \hq  & = \avg{q^1} + \delta \avg{q^2} + \dots,
\end{align}
\end{subequations}
we can write \eqref{o2avg} as
\begin{subequations} \label{o2avg-combined}
\begin{align}
    \delta(\heta_t + \hq_x) & = \order(\delta^3) \\
    \delta \left(\frac{\mean{H^{-1}}}{g}\hq_t + \heta_x \right) + \delta^2\frac{\mean{H^{-2}}}{g} \left( - \heta \ \hq_t + (\hq^2)_x \right) & = \order(\delta^3).
\end{align}
\end{subequations}
Note that superscripts on $\heta$ and $\hq$, denote exponents.  Indeed, as we will not need to work directly with the expansion \eqref{delta-series} any more,
all superscripts in the rest of the paper are exponents.
We see that $\heta, \hq$ approximately satisfy a hyperbolic equation with weak 
nonlinearity.
Using $\hq_t = -c^2 \heta_x + \order(\delta)$,
the last equation can be rewritten as
$$
    \delta ( \hq_t + c^2 \heta_x) + \delta^2
    \left( \frac{\mean{H^{-2}}}{\mean{H^{-1}}} (\hq^2)_x + g \frac{\mean{H^{-2}}}{\mean{H^{-1}}^2} \heta \ \heta_x
    \right) = \order(\delta^3),
$$
which is an averaged version of the original equation \eqref{eq:2b}.

\subsection{Higher-order homogenized equations}
Following a similar (though increasingly cumbersome) procedure, we analyze the terms
proportional to $\delta^3, \delta^4$, and $\delta^5$.  The details are not
included here as the expressions become increasingly lengthy.
The number of terms appearing in the equations at each order (before simplification) is listed in Table \ref{tbl:terms}.
We have used Wolfram \textit{Mathematica} to perform the calculations, and the notebook that contains them is
included in the supplementary material.
Properties of the operator $\fluctint{\cdot}$
proved in Appendix \ref{sec:appendix} are essential in simplifying the lengthy expressions that arise.
Summing the terms up to  $\order(\delta^5)$, we obtain the following equations
\begin{subequations} \label{o4avg-ttt}
\begin{align}
    \delta (\heta_t + \hq_x) & = \order(\delta^6) \\
    \delta (\hq_t + c^2 \heta_x) + \delta^2
    \frac{\mean{H^{-2}}}{\mean{H^{-1}}} \left( (\hq^2)_x - \heta \ \hq_t \right) \nonumber \\
    + \delta^3 \left( -\frac{\mu}{c^2} \hq_{ttt}  - \frac{\alpha_2}{c^2} \hq^2 \hq_t - \frac{\Hm{-3}}{\Hm{-1}}\heta \left(2 (\hq^2)_x - \heta \ \hq_t \right) \right) \nonumber \\
    + \delta^4 \left( \frac{\hat{\alpha}_4}{c^2} \hq^3 \hq_x  + \frac{\hat{\alpha}_6}{c^2} \hq^2 \heta \ \hq_t
    - 2 \frac{\gamma}{c^2}\left( 2 \hq_x \hq_{tt} - \heta  \ \hq_{ttt}\right) + \frac{\Hm{-4}}{\Hm{-1}}\heta^2\left(3 (\hq^2)_x - \heta \ \hq_t \right) \right) \\
    + \delta^5 (\frac{\nu_1}{c^4} \hq_{ttttt} + \frac{\nu_2}{c^2} \hq_{xxttt} + \hat{F}(\heta,\hq)) & = \order(\delta^6),
\end{align}
\end{subequations}
where $\hat{F}(\heta,\hq)$ is a function that depends on $\heta$,
$\hq$ and its derivatives (see Appendix \ref{sec:coefficients}); all terms of $\hat{F}$
are nonlinear. 
Expressions for the coefficients of this equation are also given in Appendix \ref{sec:coefficients}.
We have explicitly written out the linear 5th-order terms because these
turn out to have the most significant effect on the solution, and also
affect the linear dispersion relation as discussed below.
It can be shown that $\alpha_1<0$, $\alpha_2< 0$, and $\alpha_3\leq 0$ 
(see Remark \ref{rem:alpha2}), and also that 
$\mu>0$ (see \eqref{mu}).

\begin{table}
\center
\begin{tabular}{c|rr}
Order $\delta^p$ & $q^p_y$ terms & $\eta^p_y$ terms \\ \hline
$p=1$ & 2 & 3 \\
$p=2$ & 5 & 8 \\
$p=3$ & 14 & 37 \\
$p=4$ & 62 & 138 \\
$p=5$ & 239 & 653 \\
\end{tabular}
\caption{Number of terms appearing in the equations for $q^p_y$ and $\eta^p_y$ as a function
of the order $p$.  Terms are counted after collecting like terms (those with the same product
of $q$, $\eta$, and their derivatives) but before any further simplification.\label{tbl:terms}}
\end{table}

\subsection{Alternative forms of the homogenized equations}
The homogenized system \eqref{o4avg-ttt}, as written, is inconvenient for
two reasons.  First, the equations include high-order time derivatives, whereas the shallow water equations we started from
are first-order in time.
In principle this could be remedied by introducing additional
variables representing time derivatives of the dependent quantities.
However, a second and more serious issue is that these equations exhibit a linear instability at low wavenumbers, as explained in Section \ref{sec:analysis}.
 One way to remedy this is to convert all higher-order
time derivatives to space derivatives, as was done in \cite{leveque2003}.
This is accomplished by differentiating the equations and using equality of
mixed partial derivatives, keeping again only terms up to the desired order
in $\delta$:
\begin{subequations} \label{o4avg}
\begin{align}
    \delta(\heta_t + \hq_x) & = \order(\delta^5) \\
    \delta(\hq_t + c^2 \heta_x) + \delta^2
    \frac{\mean{H^{-2}}}{\mean{H^{-1}}}\left( c^2 \heta \ \heta_x + (\hq^2)_x \right) \nonumber \\
    + \delta^3 \left( c^2 \mu \heta_{xxx} + \alpha_1 \hq \ \heta \ \hq_x + \alpha_2 \hq^2 \heta_x + g \alpha_3 \heta^2 \heta_x \right) \nonumber \\
    + \delta^4 \bigg( \frac{\alpha_4}{g} \hq^3 \hq_x + \alpha_5 \heta^2 \hq \ \hq_x + \alpha_6 \hq^2 \heta \ \heta_x
    + g\alpha_7 \heta^3 \heta_x \nonumber \\
    + \hat{\alpha}_8 \hq_x \hq_{xx} + g \hat{\alpha}_9 \heta_x \heta_{xx} + g\hat{\alpha}_{10} \heta \ \heta_{xxx} + \hat{\alpha}_{11} \hq \ \hq_{xxx} \bigg) & = \order(\delta^5),
\end{align}
\end{subequations}
where again coefficients are provided in Appendix \ref{sec:coefficients}.
The system \eqref{o4avg} is linearly stable for small wavenumbers.
However, it unstable for sufficiently large wavenumbers (see Section \ref{sec:analysis}). This problem
is less severe, since we do not expect the homogenized equations to be
valid for large wavenumbers in any case, and solutions of \eqref{o4avg}
agree reasonably well with solutions of the original shallow water equations.

However, it is possible to obtain a system that is linearly stable for
all wavenumbers (and equivalent to the systems above up to the given
order in $\delta$).  This can be accomplished by rewriting the linear term
proportional to $q_{ttt}$ in terms of $q_{xxt}$ (again using equality
of mixed partial derivatives).  It is convenient to write the
nonlinear terms with high-order derivatives in terms of spatial
derivatives only; we are free to do this since they do not affect
the linear dispersion relation. In this way we obtain the system:
\begin{subequations} \label{o4avg-xxt}
\begin{align}
    \delta(\heta_t + \hq_x) & = \order(\delta^6) \\
    \delta\left(\hq_t + c^2 
    \heta_x\right) + \delta^2
    \frac{\mean{H^{-2}}}{\mean{H^{-1}}}\left( c^2 \heta \ \heta_x + (\hq^2)_x \right) \nonumber \\
    + \delta^3 \left( -\mu \hq_{xxt} + \alpha_1 \hq \ \heta \ \hq_x + \alpha_2 \hq^2 \heta_x + g \alpha_3 \heta^2 \heta_x \right) \nonumber \\
    + \delta^4 \bigg( \frac{\alpha_4}{g} \hq^3 \hq_x + \alpha_5 \heta^2 \hq \ \hq_x + \alpha_6 \hq^2 \heta \ \heta_x
    + g\alpha_7 \heta^3 \heta_x \nonumber \\
    + \alpha_8 \left(2 \hq_x \hq_{xx} + c^2 \heta \ \heta_{xxx} \right)+ \alpha_9 \left( 5 c^2 \heta_x \heta_{xx}  + 2 \hq \ \hq_{xxx}\right) \bigg)  \nonumber \\
    + \delta^5 \left( (\nu_1 + \nu_2 - \mu^2 ) \hq_{xxxxt} + F(\heta,\hq)\right) & = \order(\delta^6),
\end{align}
\end{subequations}
where again coefficients and the details of $F$ are provided in Appendix \ref{sec:coefficients}.


\section{Properties of the homogenized equations}
\label{sec:analysis}

\subsection{Dispersion relations}
In this section we study the dispersion relation of the systems \eqref{o4avg-ttt}, \eqref{o4avg}, and \eqref{o4avg-xxt}.
For simplicity of notation we drop the average sign on $\eta$ and $q$.

We start by considering the first form of the homogenized system \eqref{o4avg-ttt}, but neglecting the $\order(\delta^5)$ terms.
We linearize around the constant state $(\eta_0,q_0) = (0,0)$, 
obtaining the system 
\begin{subequations}\label{homog-1-lin}
     \begin{align}
      \eta_t + q_x & =  0, \\ 
      q_t + c^2 \eta_x - \frac{\muhat}{c^2} q_{ttt} & =  0,
      \end{align}
\end{subequations}
where $\muhat = \delta^2\mu$.
Now we look for solutions of system \eqref{homog-1-lin} as a superposition of Fourier modes of the form 
$\eta(x,t) = \hat{\eta} \exp(i(kx-\omega t))$, 
$q(x,t) = \hat{q} \exp(i(kx-\omega t))$, 
 where $i$ denotes the imaginary unit.
Inserting this {\em ansatz\/} in \eqref{homog-1-lin} we obtain:
\begin{subequations}\label{ansatz}
     \begin{align}
     -i\omega \hat{\eta} + i k \hat q & = 0\\
     - i \omega \hat{q} + i k c^2 \hat \eta - i \frac{\muhat}{c^2} \omega^3  \hat{q}& = 0. 
      \end{align}
\end{subequations}
Non-trivial solutions of \eqref{ansatz} are obtained if
the determinant of the coefficient matrix is zero, which is equivalent to
\begin{equation}\label{disp_rel_3_1}
\omega^2 + \frac{\muhat}{c^2} \omega^4 - c^2k^2 = 0.
\end{equation}
This equation is the dispersion relation for the linearized system
\eqref{homog-1-lin}. In order to reduce the number of parameters, we divide this relation by $c^2k^2$, obtaining the dispersion relation in non-dimensional form:
\begin{equation}
    \Omega^2 + K^2\Omega^4 - 1 = 0,
    \label{disp-rel-1_nondim}
\end{equation}
where 
\[
    \Omega = \frac{\omega}{ck}, \quad K = k \sqrt{\muhat}.
\]
Note that $\muhat>0$ so that $K$ is real.
System \eqref{homog-1-lin} requires four initial conditions, namely one has to assign 
$\eta(x,0), q(x,0), q_t(x,0),q_{tt}(x,0)$, which correspond to four initial values for each Fourier mode. This means that for each wave number $k$ there will be four linearly-independent solutions, one for each solution of the biquadratic equation 
\eqref{disp-rel-1_nondim}. The four roots are given by the two equations 
\[
    \Omega^2 = Z_+, \quad \Omega^2 = Z_-
\]
where 
\[
    Z_{\pm}  = \frac{-1\pm \sqrt{1+4K^2}}{2K^2}
\]
therefore  the two roots of $\Omega^2 = Z_+$ 
will be real and the two roots of
$\Omega^2=Z_-$ 
will be imaginary conjugate. 
The existence of a root with positive coefficient of the imaginary part indicates that the initial value problem for system \eqref{homog-1-lin} is ill-posed. 
For this reason the homogenized systems in the form  
\eqref{o4avg-ttt} are not very useful for any practical purpose. 

We now consider the second form of homogenized system, namely \eqref{o4avg}.
Linearizing around $(\eta_0,q_0) = (0,0)$ we obtain the system 
\begin{subequations}\label{homog-2-lin}
     \begin{align}
      \eta_t + q_x & =  0, \\ 
      q_t + c^2 \eta_x + c^2 \muhat \eta_{xxx} & =  0.
      \end{align}
\end{subequations}
Using the same {\em ansatz\/} as before, i.e., $\eta(x,t) = \hat{\eta} \exp(i(kx-\omega t))$, inserting it in system 
\eqref{homog-2-lin}, and looking for non-trivial solutions gives the following dispersion relation
\[
    \omega^2-c^2k^2+\muhat c^2 k^4 = 0,
\]
or in non-dimensional form
\begin{equation}
    \Omega^2+K^2 = 1.
    \label{disp-rel-2_nondim}
\end{equation}
Solving for $\Omega$ we find
\[
    \Omega_\pm = \pm \sqrt{1-K^2}.
\]
This means that the dispersion relation corresponding to the second form of the system allows only bounded wavenumbers: the non-dimensional form requires $|K|\leq 1$, which corresponds to 
\[
    |k|\leq k_{\rm max} \equiv \frac{1}{\delta\sqrt{\mu}}.
\]
Since we do not expect the homogenized approximation to be accurate for such large wavenumbers, it is possible to work with
this form of the equations, as is done for
instance in \cite{leveque2003}.
However, we prefer the third form of
the equations, given in \eqref{o4avg-xxt}
for reasons that will soon be apparent.

The third form of the homogenized model equation is reported in \eqref{o4avg-xxt}.
Once again neglecting terms of $\order(\delta^5)$ and
linearizing around $(\eta_0,q_0) = (0,0)$, we obtain
\begin{subequations}\label{homog-3-lin}
     \begin{align}
      \eta_t + q_x & =  0, \\ 
      q_t + c^2 \eta_x - \muhat q_{xxt} & =  0.
      \end{align}
\end{subequations}
Using the same {\em ansatz\/} as before, i.e., $\eta(x,t) = \hat{\eta} \exp(i(kx-\omega t))$, inserting it in system 
\eqref{homog-3-lin}, and looking for non-trivial solutions gives the dispersion relation
\[
    \omega^2-c^2k^2+\muhat\omega^2 k^2 = 0,
\]
which can be written in non-dimensional form after dividing by $c^2k^2$:
\begin{equation}
    \Omega^2(1+K^2) - 1 = 0.
    \label{disp-rel-3_nondim}
\end{equation}
Solving for $\Omega$ one obtains 
\[
    \Omega_\pm = \pm \frac{1}{\sqrt{1+K^2}},
\]
therefore $\forall K\in \mathbb{R}$ we have $\Omega\in\mathbb{R}$, therefore there are no unstable modes or forbidden wave numbers. 

A plot of the dispersion relation corresponding to the positive branches of relations \eqref{disp-rel-2_nondim} and \eqref{disp-rel-3_nondim},
and to the real positive branch of \eqref{disp-rel-1_nondim} is shown in Figure \ref{fig:disp_rel_3}.
\begin{figure}
\begin{center}
    \includegraphics[width=3.5in]{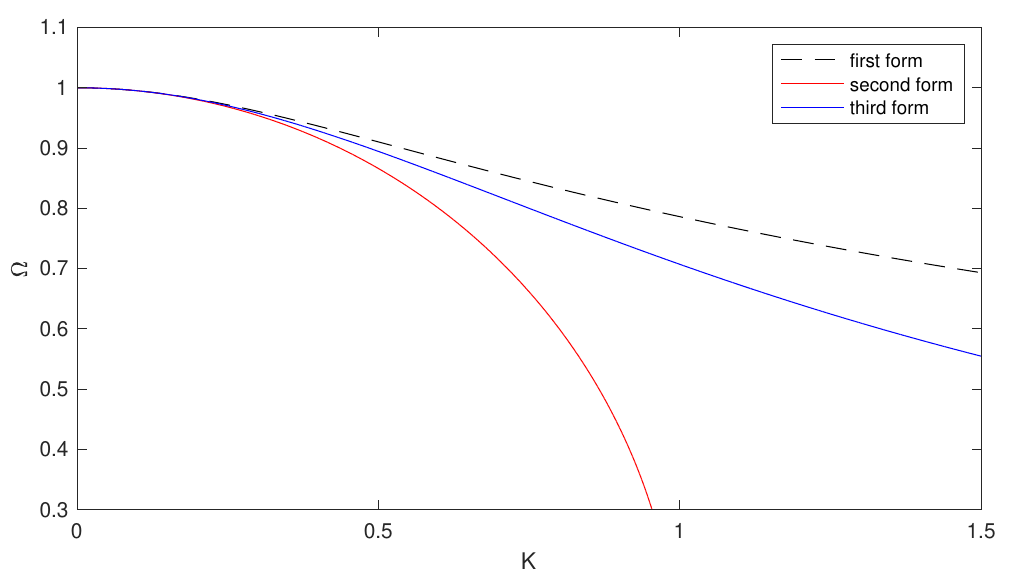}
    \caption{Non-dimensional dispersion relations corresponding to the three forms of the linearized homogenized systems \eqref{o4avg-ttt}, \eqref{o4avg}, and \eqref{o4avg-xxt}. 
    \label{fig:disp_rel_3}}
\end{center}
\end{figure}
Notice that, as expected, for small values of $K$, all branches give
\[
    \Omega = 1-\frac{K^2}{2} + \order(K^4),
\]
which corresponds to the dispersion relation $\omega = ck-\frac12 \muhat c k^3+\order(k^5)$.

Finally, we consider the linearization of system \eqref{o4avg-xxt} \emph{including} the $\order(\delta^5)$ terms. The linearized system reads
\begin{subequations}\label{homog-5-lin}
     \begin{align}
      \eta_t + q_x & =  0, \\ 
      q_t + c^2 \eta_x - \muhat q_{xxt} + (\hat{\nu}^2-\hat{\mu}^2)q_{xxxxt}& =  0,
      \end{align}
\end{subequations}
where $\hat{\nu}^2 = \delta^4(\nu_1+\nu_2)$.
Proceeding as before, we obtain the dispersion relation 
\[
    \omega^2(1+\hat{\mu}k^2+(\hat{\nu}^2-\hat{\mu}^2) k^4)-c^2k^2 = 0.
\]
Using the notation defined above, the relation can be written as 
\[
    \Omega_\pm = \pm \frac{1}{\sqrt{1+K^2+r K^4}}
\]
where $r=(\nu_1+\nu_2)/\mu^2-1$. 
We see that all wavenumbers are stable if
$\nu_1+\nu_2>\mu^2$; this condition is fulfilled in
all examples we have studied, including those shown in
the next section.  Indeed, it can be proved that this condition is satisfied
for any piecewise-constant bathymetry; see Appendix \ref{sec:pwc-coeff}.


\subsection{Traveling wave solutions}
All three model systems, \eqref{o4avg-ttt}, \eqref{o4avg} and \eqref{o4avg-xxt}, admit traveling wave solutions; this is the case whether or not we include the 4th- and 5th-order terms.
Here we describe how to construct periodic solutions of system \eqref{o4avg-xxt}. 
We focus on this form of the equations since it is the most amenable for numerical simulations.

\subsubsection{Traveling wave to $\order(\delta^3)$}
We start describing how to compute a traveling wave for system \eqref{o4avg-xxt}, 
accounting for terms up to $\order(\delta^3)$. To this purpose we adopt  
a common technique that can be found in textbooks such as \cite{drazin1989solitons}.
We start by rewriting system \eqref{o4avg-xxt} in the form 
\begin{subequations}\label{o3avg-xxt}
     \begin{align}
      \eta_t + q_x & =  0 \\ 
      q_t + c^2\eta_x + \bhat_1 \eta\eta_x + \bhat_2 (q^2)_x - \bhat_3 q\eta q_x - \bhat_4 q^2\eta_x - \bhat_5\eta^2\eta_x - \hat{\mu}\, q_{xxt} & = 0, 
      \end{align}
\end{subequations}
where 
\[
\bhat_2 = \delta\frac{\Hm{-2}}{\Hm{-1}}, \quad \bhat_1 = \bhat_2 c^2, \quad \bhat_3 = -\delta^2\alpha_1,\quad
\bhat_4 = -\delta^2\alpha_2, \quad \bhat_5 = -g\delta^2\alpha_3, \quad \hat{\mu} = \delta^2\mu.
\]
Now we look for a solution that depends only on $\xi\equiv x-Vt$, where $V$ is the desired traveling wave speed. 
Assuming then $\eta = \eta(\xi), \> q = q(\xi)$, the first equation becomes 
\[
     -V \eta' + q'  =  0,
\]
where the prime denotes differentiation with respect to $\xi$. This equation immediately tells us that $q = q_0 + V(\eta-\eta_0)$. 
We consider waves propagating on a still, flat-surface background, so that as $|x|\to\infty$ we have $q=0$ and we can choose the vertical reference point so that $\eta=0$ in the unperturbed state.  Thus we take $\eta_0=q_0=0$.
Replacing $q$ by $V\eta$, the second equation becomes:
\begin{equation}\label{o3avg-xxt-travel}
      -V^2\eta' + c^2\eta' + \bhat_1 \eta\eta' + V^2\bhat_2 (\eta^2)' - V^2\bhat_3 \eta^2 \eta' 
      - \bhat_4 V^2\eta^2\eta' - \bhat_5\eta^2\eta' + V^2 \hat{\mu}\, \eta''' = 0.
\end{equation}
Observing that $\eta\eta' = (\eta^2)'/2$, and $\eta^2\eta' = (\eta^3)'/3$, after changing sign to all terms in \eqref{o3avg-xxt-travel}, 
the equation can be written as 
\[
    \frac{d}{d\xi}( \gamma_1\eta -\gamma_2\eta^2+\gamma_3\eta^3 - V^2\hat{\mu}\,\eta'') = 0,
\]
where 
\[
    \gamma_1 = V^2-c^2, \quad \gamma_2 = \frac12\bhat_1 + V^2 \bhat_2, \quad \gamma_3 = \frac13((\bhat_3 + \bhat_4)V^2+\bhat_5).
\]    
We can integrate once to obtain
\begin{equation}\label{eq:Newton}
    \eta'' = F(\eta), \quad \textrm{with} \> F(\eta) \equiv (\gamma_1\eta -\gamma_2\eta^2+\gamma_3\eta^3 + A)/(\hat{\mu}\,V^2),
\end{equation}
where $A$ is an integration constant. 
Again, since $\eta\to 0$ as $|x|\to\infty$, we have $A=0$.
Eq.~\eqref{eq:Newton} can be interpreted as Newton's second law of a particle of unit mass under the action of a positional force $F(\eta)$.
Here $\eta$ plays the role of the particle position and $\xi$ represents the ``time''. 
This equation admits a first integral, which plays the role of the total energy:
multiplying Eq.~\eqref{eq:Newton} by $\eta'$ and integrating once more one obtains
\begin{equation}\label{eq:energy}
    \frac12(\eta')^2+U(\eta) = E
\end{equation}
with 
\begin{equation}\label{eq:potential}
    U(\xi) = \left(-\frac12\gamma_1\eta^2+\frac13\gamma_2\eta^3-\frac14\gamma_4\eta^4\right)/(\hat{\mu}\,V^2). 
\end{equation}
A typical shape of the potential \eqref{eq:potential} corresponding to $\gamma_i>0, i=1,\ldots,3$ is shown 
in the top-left panel of 
Fig.~\ref{fig:potential}, where the traveling waves 
are obtained from the dynamics of a point that moves in a potential well. Level set curves of constant ``energy'' are represented in the bottom-left panel. 
\begin{figure}
    \begin{center}
    \includegraphics[width=2.5in]{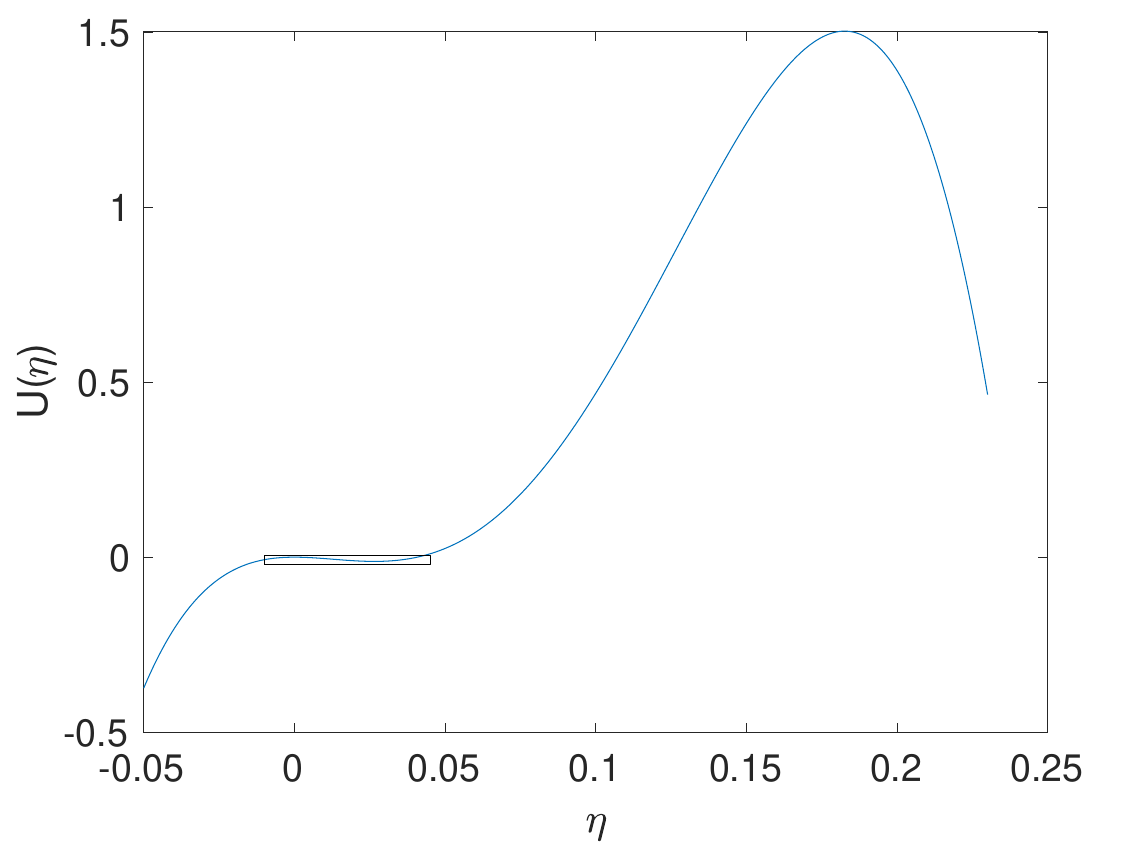}
    \includegraphics[width=2.5in]{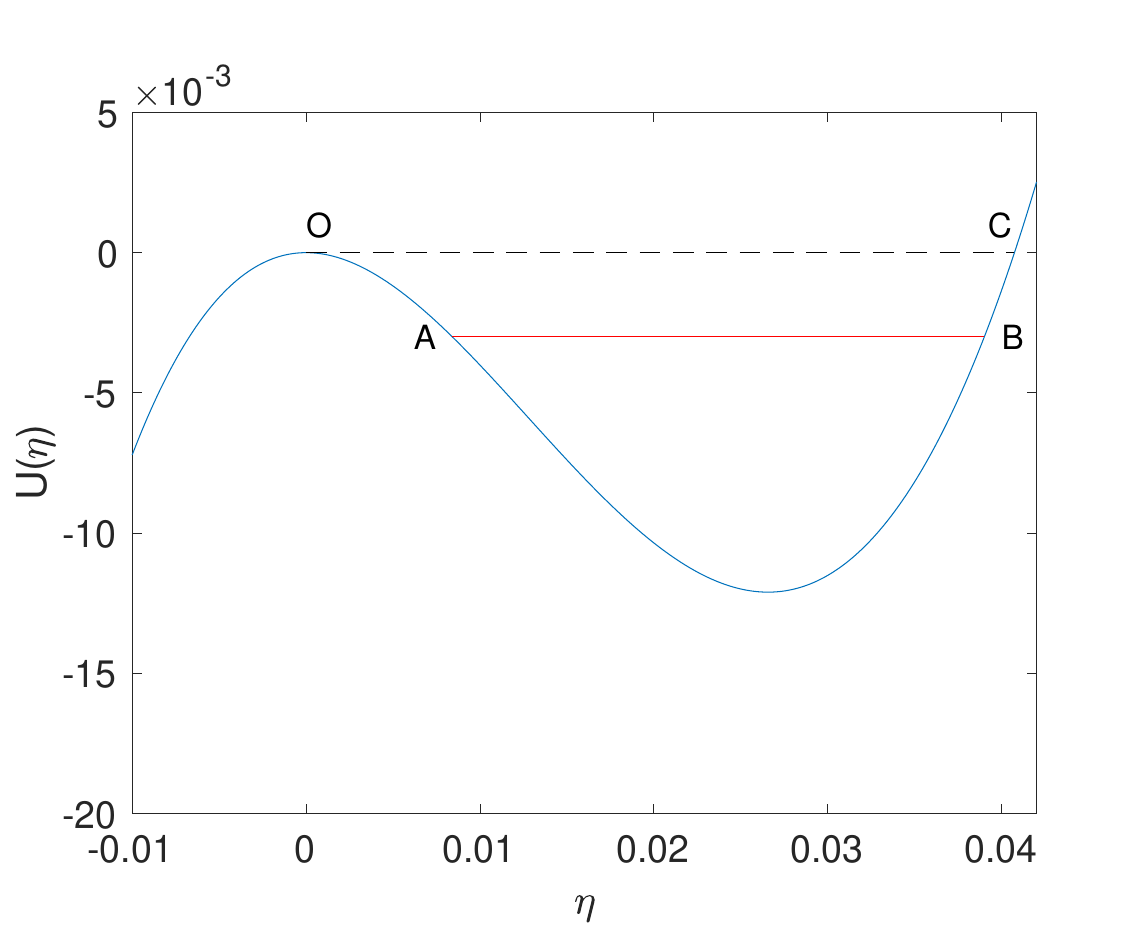}
    \includegraphics[width=2.5in]{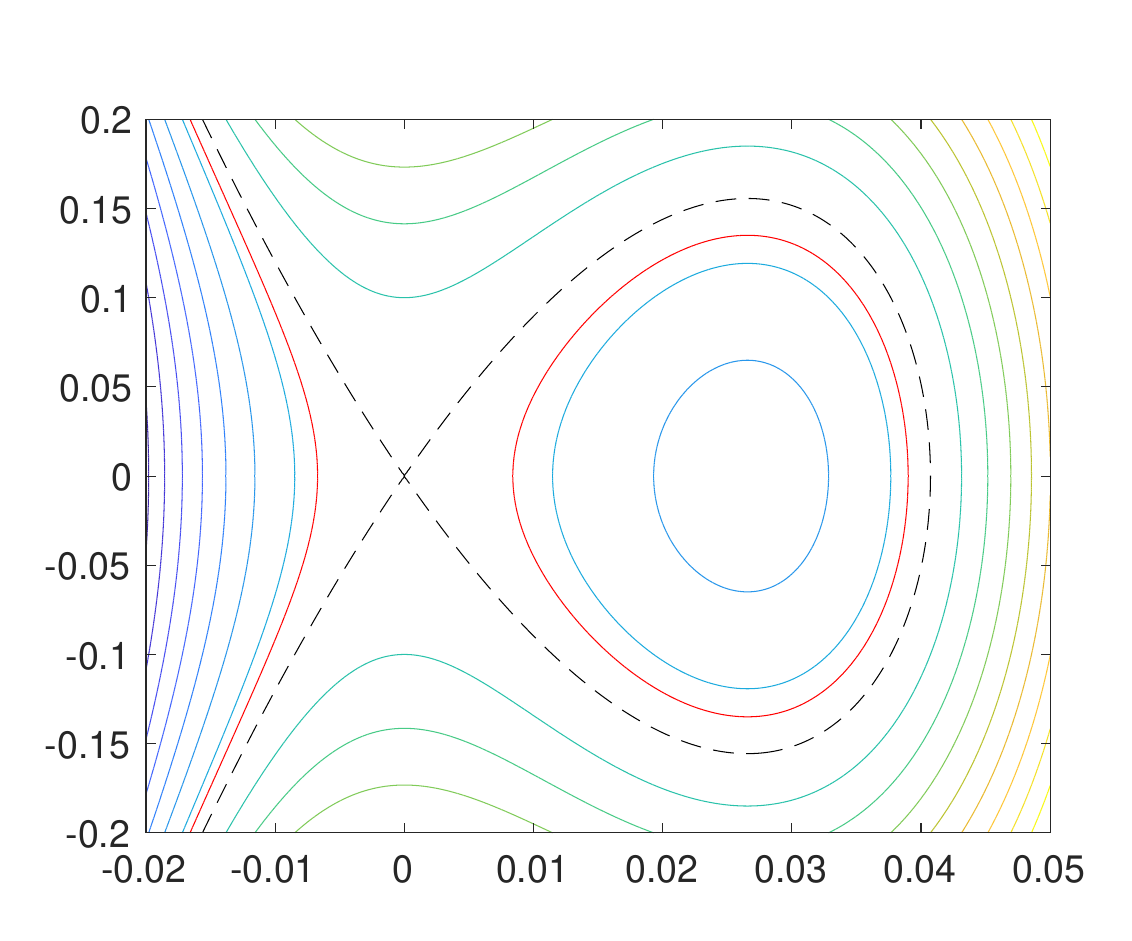}
    \includegraphics[width=2.5in]{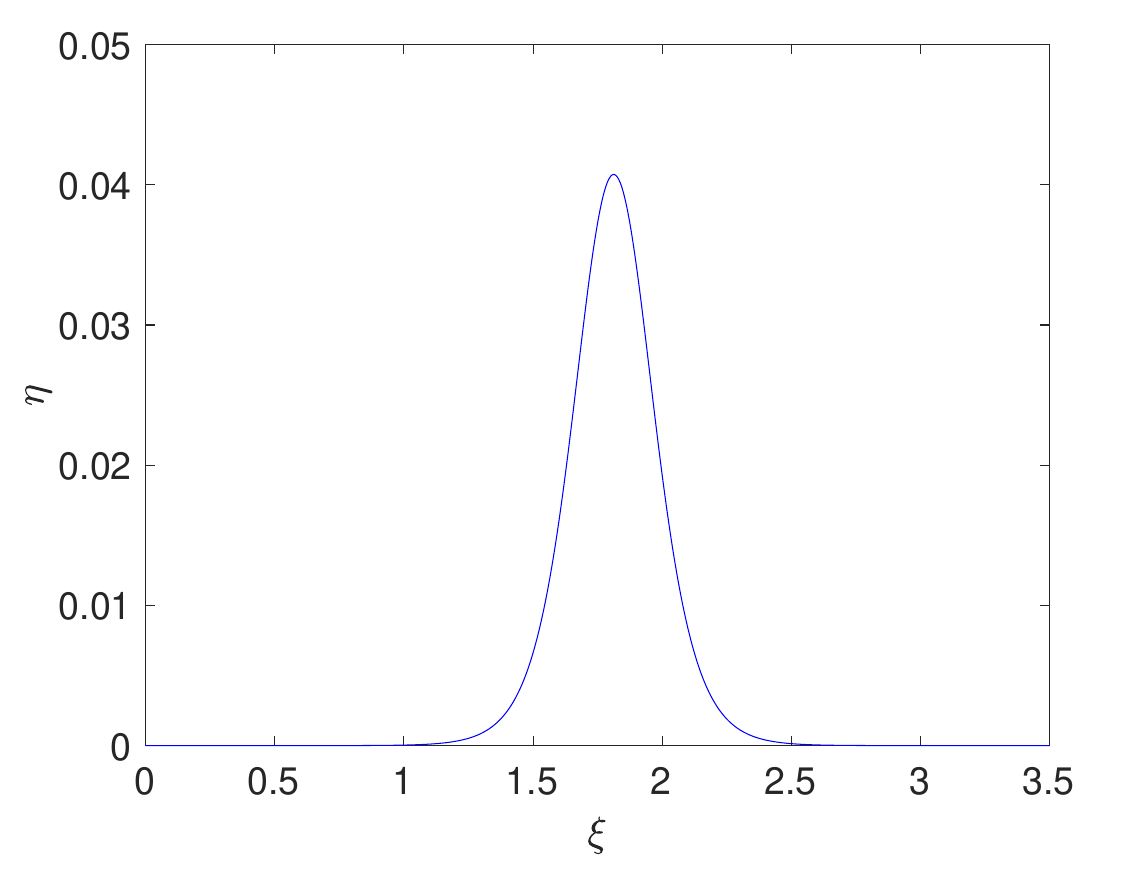}  
    \end{center}
    \caption{Typical potential for the ``motion" of the particle. The first panel represents a typical potential $U(\xi)$ of Eq. \eqref{fig:potential}, while the second panel is a closeup of the first one in the area near the origin. If the total energy is less than zero and the initial position of the particle is located between $0$ and $\eta_C$ (second panel), then the motion is periodic. This case is described by the red curve, which corresponds to a ``particle'' oscillating between points $A$ and $B$. $\eta_A$ and $\eta_B$ represent the minimum and the maximum of the periodic traveling wave, while the period is twice the ``time'' necessary to go from $A$ to $B$. If the level of the energy approaches zero from below, then the period becomes longer and longer, up to the limit situation corresponding to the dashed line. In such a case the solution describes a solitary wave traveling with speed $V$. The third panel  displays the trajectory in phase space. The red energy level corresponds to the periodic solution which oscillates between $A$ and $B$. The dashed line represents the separatrix, whose right lobe corresponds to a solitary wave, which is reported in the right panel. Finally, the last panel shows the shape of the solitary wave as a function of $\xi=x-Vt$.}\label{fig:potential}
\end{figure}
Integrating the equation of motion \eqref{eq:Newton} with a point in phase space which lies inside the right dashed lobe one obtains an orbit $(\eta(\xi),\eta'(\xi))$. The expression of $\eta(\xi)$ corresponds to a periodic traveling wave that moves with speed $V$.
The solitary wave is a particular orbit that passes through the separatrix. It can be viewed as 
the limiting case of a periodic orbit as the total energy approaches zero, 
and the corresponding period tends to infinity. The panel on the right depicts the solitary wave corresponding 
to the separatrix. 
The separatrix has been computed by a symplectic integrator based on a Gauss--Legendre collocation method for the solution of the initial value problem for a second order ordinary differential equation of the form \eqref{eq:Newton}, using the package {\tt gni\_irk2} \cite{hairer2003gnicodes}; see also 
\cite{Henrier_structure_preserving}.
As initial condition we take a point very close to the origin, with $\eta(0) = 10^{-9}$ and 
$\eta'(0) = \sqrt{\gamma_1/\tilde{\mu}}\,\eta(0)$, where $\tilde{\mu} = \mu(V/c)^2$.

\subsubsection{Traveling wave to $\order(\delta^5)$}
Now we consider the more challenging problem of computing a traveling wave for system 
\eqref{o4avg-xxt}.
Using the same {\em ansatz\/} adopted in the previous section, i.e., looking for a solution of the form 
$\eta = \eta(x-Vt)$ and $q = q(x-Vt)$, we obtain, as before, $q=V\eta$. Replacing this in the second equation, one obtains the following equation for $\eta$:
\[
    -\gamma_1\eta' + 2\gamma_2 \eta\eta' - 3\gamma_3 \eta^2\eta' + 4\gamma_4\eta^3\eta' 
    + 2\gamma_5\eta'\eta'' + 2\gamma_6\eta\eta''' + \hat{\mu}\eta'''-\hat{\nu}\eta^{(5)} = 0.
\]
Integrating this equation we obtain 
\begin{equation}\label{eq:fourth}
    -\gamma_1\eta + \gamma_2 \eta^2 - \gamma_3 \eta^3 + \gamma_4\eta^4 
    + \gamma_5(\eta')^2 + \gamma_6(2\eta\eta'' - (\eta')^2) + \hat{\mu}\eta''-\hat{\nu}\eta^{(4)} = A,
\end{equation}
where $A$ is a constant. Since we look for solutions that are perturbation of a constant state, 
and denote $\eta$ the deviation from such a constant state, it follows that $A=0$.
The values of the additional coefficients are given by 
\[
    \gamma_4 = \frac{\delta^3}{4}
    \left(\frac{\alpha_4}{g}\, V^4 + \alpha_5 \, V^2 + \alpha_6 \, V^2 + g \alpha_7 \right),\quad
    \gamma_5 = \delta^3\left(\alpha_8 \, V^2 + \frac52\alpha_9 c^2\right), 
\]
\[
    \gamma_6 = \delta^3\left(\frac12\alpha_8 c^2+\alpha_9\,V^2\right),\quad
    \hat{\nu} = \delta^4V^2(\nu_1+\nu_2-\mu^2).
\]
In this case the usual analogy with classical mechanics is not possible. Furthermore, the system does not seem to possess further first integrals. It cannot be written as a Hamiltonian system of a particle in 
$\mathbb{R}^2$, since the ``forces'' are not gradient of a potential. 
We therefore resort to a completely different technique in order to find the solution of the ODE 
\eqref{eq:fourth} that corresponds to a traveling wave. 
We consider a domain which is sufficiently wide to contain the solitary wave, and such that the expected value of the traveling wave at the boundary of the domain is less than $10^{-7}$.
Then we discretize the domain, approximate the derivatives by finite difference and solve a boundary-value problem with homogeneous Neumann and Dirichlet conditions at both ends, looking for non-trivial solutions. 
Notice that $\eta \equiv 0 $ is a trivial solution of \eqref{eq:fourth}. However, since the equation is nonlinear, it may have more solutions. 
Using Newton's method with the
$\order(\delta^3)$ solitary wave as initial guess, we obtain the solution shown by a dashed line in 
Fig.~\ref{fig:compare_soliton}. 

Here we also plot a solitary wave obtained from the finite volume simulation of the shallow water equations (the solitary wave is plotted for a fixed value $x$, over time).
Notice that the $\order(\delta^5)$ solitary wave is in much better agreement with the true solitary wave, compared to the 
$\order(\delta^3)$ one.
In computing both of the homogenized solitary waves, we have chosen the velocity $V$ in order to yield an amplitude equal to that of the true solitary wave. 
This required choosing slightly different velocities in the two cases. Specifically, we take 
$V_3 = c\times 1.023928$ for the $\order(\delta^3)$ solitary wave, and  $V_5 = c\times 1.02327$, so the solitary wave speed 
differs by less than 2.4\% from the unperturbed wave speed, and the relative difference in the imposed traveling wave speeds is about $0.064\%$.
\begin{figure}
\centering
    \includegraphics[width=3in]{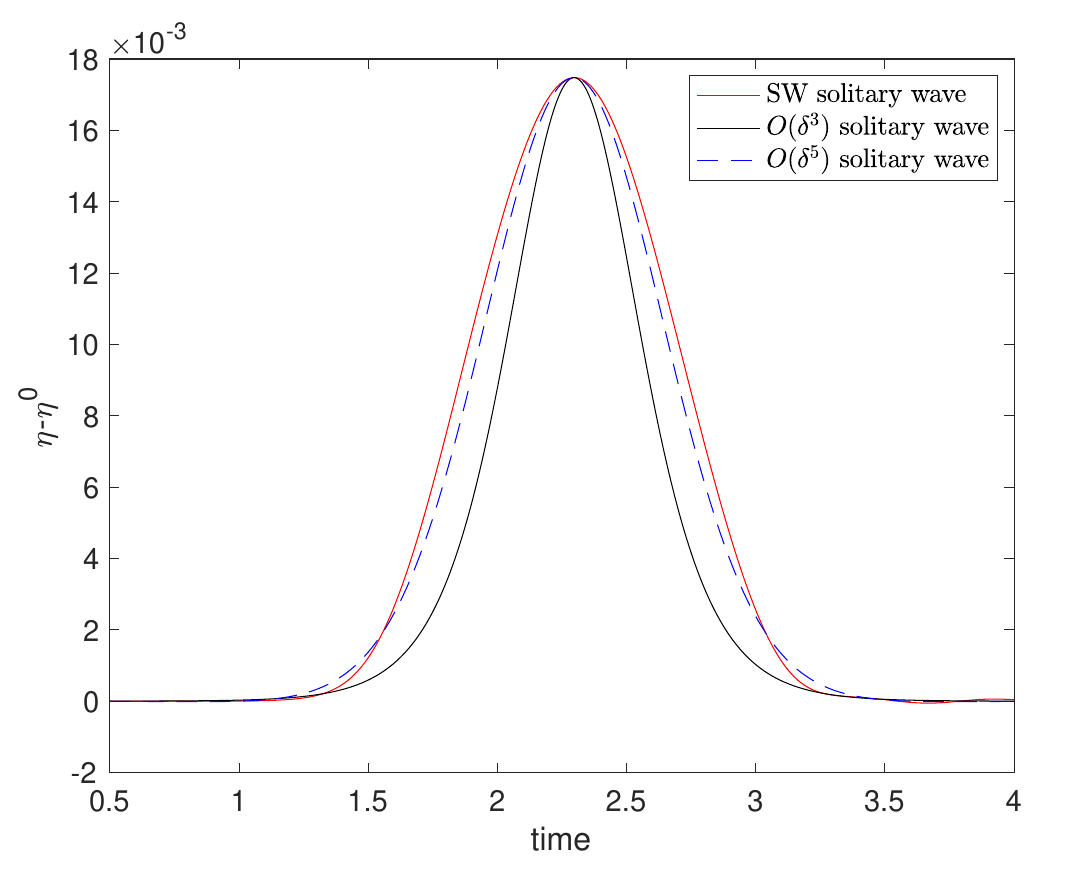}
    \caption{
 Comparison between the direct solution of the shallow water (SW) equations and the traveling wave solutions of the homogenized equations. 
    Solutions are plotted as a function of time: the red line represents the value of $\eta(\bar{x},t-\tau)$. The parameter $\tau$ is adjusted to align the computed solution with the traveling wave solutions. 
    The black line is the $\order(\delta^3)$-accurate solitary wave obtained by numerically integrating Eq.~\eqref{eq:Newton}
    along the separatrix, while the blue dashed line is accurate to $\order(\delta^5)$, obtained
    by solving Eq.~\eqref{eq:fourth}.}\label{fig:compare_soliton}
\end{figure}
    
\section{Comparison of direct and homogenized solutions \label{sec:numerical}}
In this section we explore the accuracy of the  homogenized approximation by comparing its numerical solutions to accurate solutions of the original system \eqref{eq:sw1}. 
We start by discussing the methods adopted for the numerical solution of both the original system 
\eqref{eq:sw1} and the homogenized 
system
\eqref{o4avg-xxt}.

\subsection{Numerical discretization of the homogenized equations}
We solve the homogenized equations with a Fourier
pseudo-spectral method.
We use the mixed-derivative form of the equations \eqref{o4avg-xxt},
due to its favorable dispersion relation and also its
convenience in terms of numerical discretization.
We can write this system as
\begin{align} \label{semi-discrete}
    \heta_t & = - \hq_x \\
    \hq_t & = (1-\delta^2 \mu \partial_x^2)^{-1} \FF(\heta,\hq),
\end{align}
where $\FF$ is a nonlinear operator that involves
derivatives with respect to $x$ only.  We discretize
$\FF$ in the standard pseudo-spectral way and then apply
the inverse elliptic operator $(1-\delta^2 \mu \partial_x^2)^{-1}$
in Fourier space, which does not require the solution of any
algebraic system.  The operator $\FF$ may appear to
be stiff since it contains high-order spatial derivatives.
However, these high-order terms have small coefficients
and are nonlinear, which renders them relatively small
for the weakly nonlinear regime in which we are interested.
We can therefore integrate the pseudo-spectral 
semi-discretization of \eqref{semi-discrete} efficiently with an 
explicit Runge--Kutta method.
In the examples that follow, we use the fifth-order method of Bogacki and Shampine \cite{bogacki1996}.

For the $\order(\delta^5)$ approximation, there are additional
linear terms, with 5th-order derivatives, in the momentum equation.  We convert these derivatives to be 4th-order in space and first order in time, and incorporate them
into the elliptic operator. This results in a system of the form
\begin{align} \label{semi-discrete5}
    \heta_t & = - \hq_x \\
    \hq_t & = (1-\delta^2 \mu \partial_x^2 + \delta^4 \mu_5 \partial_x^4)^{-1} \mathcal{F}_5(\heta,\hq).
\end{align}
The operator $F_5$ 
contains derivatives of up to fourth order,
but again its stiffness is alleviated by the fact that
these terms are highly nonlinear and have small coefficients.
Thus, explicit time integration is efficient in this case also.
For the spatial domain, we take $x\in[-L,L]$ where $L$ is chosen
large enough that the waves do not reach the boundaries before
the final time.

\subsection{Numerical methods for the variable-bathymetry shallow water system}
For the solution of the first-order variable-coefficient
hyperbolic shallow water system \eqref{eq:sw1} we use
the finite volume code Clawpack.
We make use of two different algorithms implemented in
Clawpack:
\begin{itemize}
    \item The \emph{classic Clawpack} algorithm, based on the Lax--Wendroff method with limiters \cite{LeVeque-FVMHP};
    \item The \emph{SharpClaw} algorithm, based on 5th-order WENO
    reconstruction in space and 4th-order Runge--Kutta integration in time \cite{KetParLev13}.
\end{itemize}
Accurate solution of this system requires a much finer spatial
grid, in order to resolve the bathymetric variation and its effects.
In order to save computational effort in the finite volume solutions, we take the spatial domain $[0,L]$
with a reflecting (solid wall) boundary condition
applied at $x=0$.

\subsection{Accuracy and computational cost}\label{sec:compcost}
The main purpose of the computations shown here is to provide a visual comparison of the homogenized model solution and the shallow water solution.  We therefore compute solutions that are sufficiently accurate so that further grid refinement would not produce any visible change in the plots.  In practice we have overdone this by computing solutions that change by less than $10^{-5}$ (absolute pointwise difference) when the grid is refined by a factor of two in both space and time.

With piecewise-constant bathymetry and a domain of length 800, reaching this level of error with a pseudospectral solver for the homogenized equations requires only 6 seconds of computation on an M1 Apple Macbook.  Using SharpClaw to solve the shallow water equations, reaching this level of error requires about 20 minutes of computation.
Thus the use of the homogenized model allows for a speedup of about 200x. Although computing the solution for a single scenario is quite feasible with either approach, the homogenized model is invaluable when exploring different parameter regimes to understand general solution behavior across a range of scenarios.

\subsection{Piecewise-constant bathymetry}
We consider first a slight variation on the example
presented in the introduction, with bathymetry
and initial data given by
\begin{subequations} \label{scenario_a}
\begin{align}
    b(x) & = \begin{cases} -1 & 0 \le x - \lfloor x \rfloor < 1/2  \\
    -0.3 & 1/2 \le x - \lfloor x \rfloor < 1, \end{cases} \\
    \eta(x,0) & = \frac{1}{40}e^{-x^2/9}, \\
    u(x,0) & = 0.
\end{align}
\end{subequations}

In Figure \ref{fig:scenario_a}, we compare the pseudospectral solution of the homogenized approximation, calculated up to terms of order three, four,
or five,
and the finite volume solution of the variable-bathymetry shallow water system \eqref{eq:sw1}.
The solution of the $\order(\delta^4)$ system is not shown because it is nearly indistinguishable from that of the $\order(\delta^3)$ system.

We see that the homogenized solution is a good approximation at early times
and becomes less accurate at later times, as expected.  The fourth-order terms provide only a small improvement.

\begin{figure}
    \includegraphics[width=\textwidth]{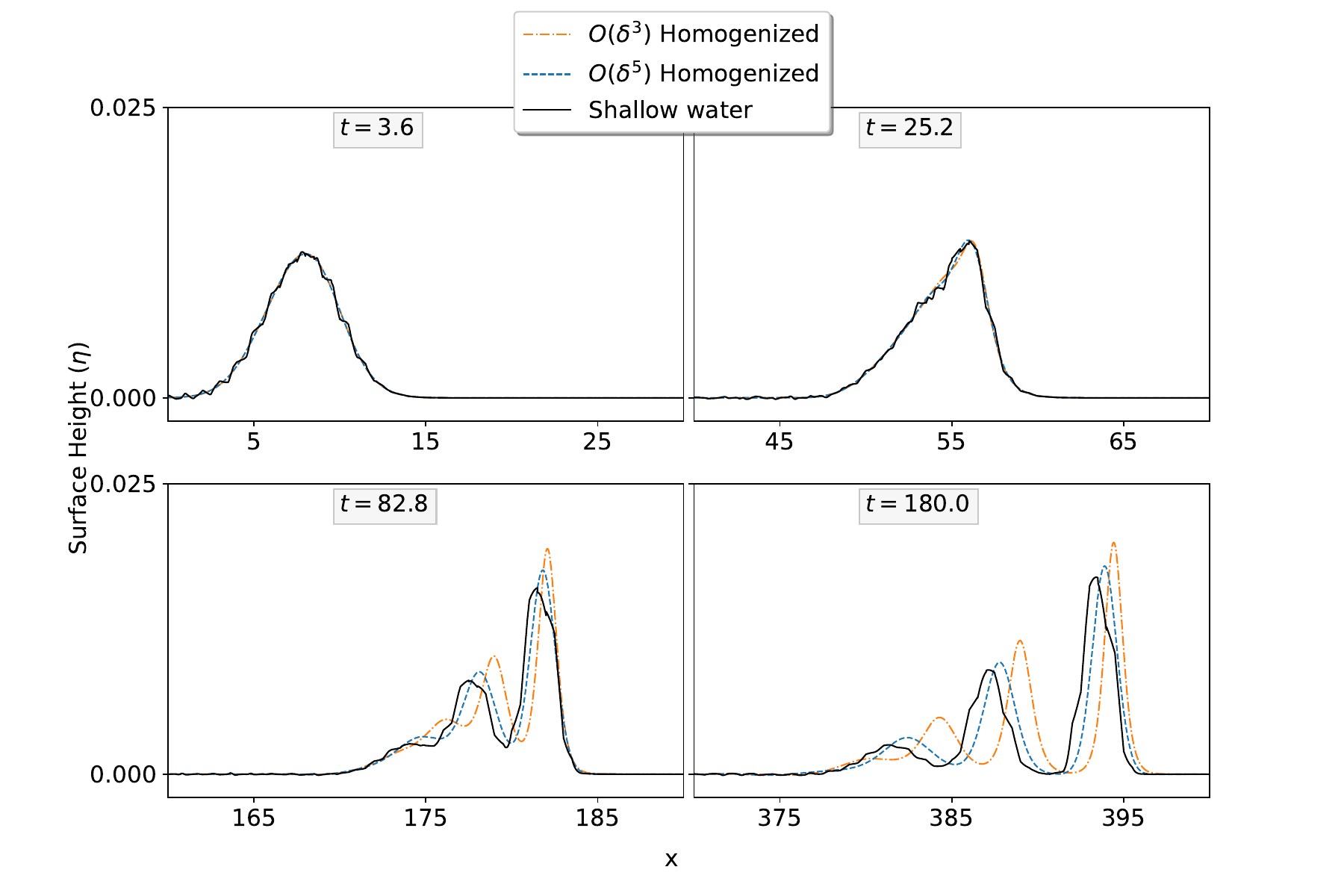}
    \caption{Comparison of homogenized and direct solutions, with
        bathymetry and initial data given by \eqref{scenario_a}.
        The surface elevation $\eta - \eta^0$ is shown, with the $x$-axis shifted to show the wave structure at each time.\label{fig:scenario_a}}
\end{figure}

Indeed, we have found that the linear dispersive terms appearing at odd
orders contribute much more significantly to the solution behavior, whereas
the nonlinear terms contribute significantly less.  Motivated by this,
we have additionally computed just the linear terms appearing in the
momentum equation at 7th order (as might be expected, there are no linear
dispersive 6th-order terms).  If these are included, the homogenized solution
becomes even closer to the shallow water solution.

It is possible to recover also the fast-scale variation in the solution
using the equations for $\eta^j(x,y,t)$ and $q^j(x,y,t)$.
The fast-scale components can be computed simply as a post-processing
step, since they are functions of $\heta$ and $\hq$ (and their derivatives)
and of $H(y)$.  In Figure \ref{fig:fast-scale} we show the solution
obtained using the $\order(\delta^5)$ averaged equations plus the linear 7th-order terms, and then adding
the fast-scale terms from \eqref{eta2} and \eqref{eta3}.  It can be seen
that these terms capture most, though not all, of the finer-scale
oscillations in the direct solution.
\begin{figure}
    \center
    \includegraphics[width=\textwidth]{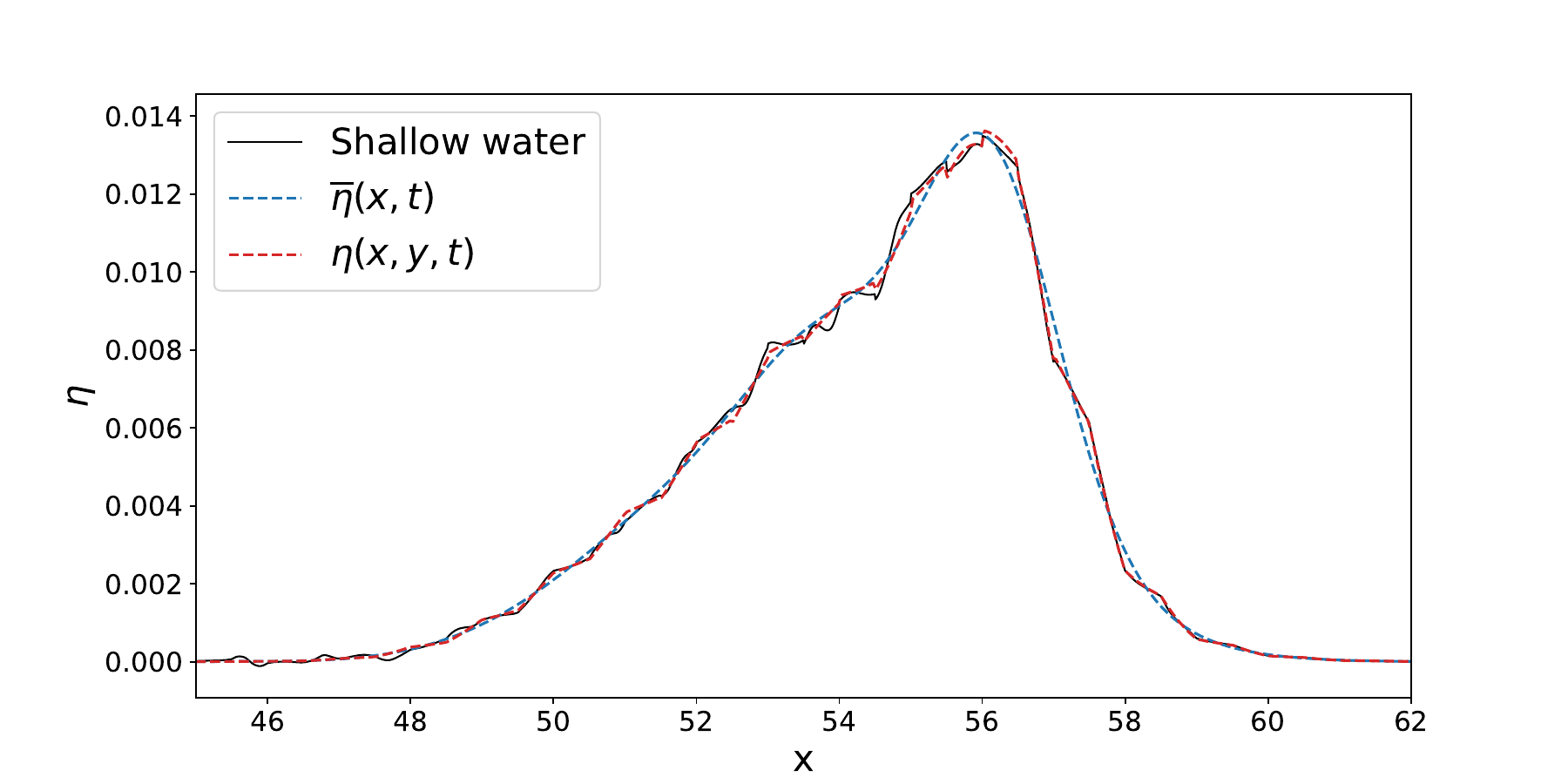}
    \caption{Comparison of homogenized and direct solutions at $t=25.2$, with
        bathymetry and initial data given by \eqref{scenario_a}.
        The surface elevation $\eta - \eta^0$ is shown.  Both the averaged and full solutions of the homogenized equations are plotted. \label{fig:fast-scale}}
\end{figure}

\subsection{Smooth bathymetry}
Next we take the same initial data as in \eqref{scenario_a}, but we use the bathymetry profile
\begin{align}\label{scenario_b}
    b(x) & = -\frac{3}{5} + \frac{2}{5} \sin (2 \pi x).
\end{align}
Numerical solutions are shown in Figure \ref{fig:scenario_b}.  We see that in this case the direct finite volume solution includes some high-wavenumber oscillations; convergence tests with different numerical methods indicate that these oscillations are part of the exact solution.  
The qualitative behavior is similar to that of the piecewise-constant case above: the accuracy of the homogenized approximation gradually 
diminishes over time, and the 5th-order approximation is more accurate than what is obtained with only the terms of up to 3rd or 4th order.

In these examples we have deliberately chosen
bathymetry and initial data that lead to solitary wave formation.
For much larger initial data, or bathymetry with smaller variation
relative to the average fluid depth, wave breaking occurs and the homogenized
model fails to fully capture the solution dynamics.
We do not investigate this further here but refer the reader to \cite{2012_ketchesonleveque_periodic,2019_periodic} for related work on this topic.

\begin{figure}
    \includegraphics[width=\textwidth]{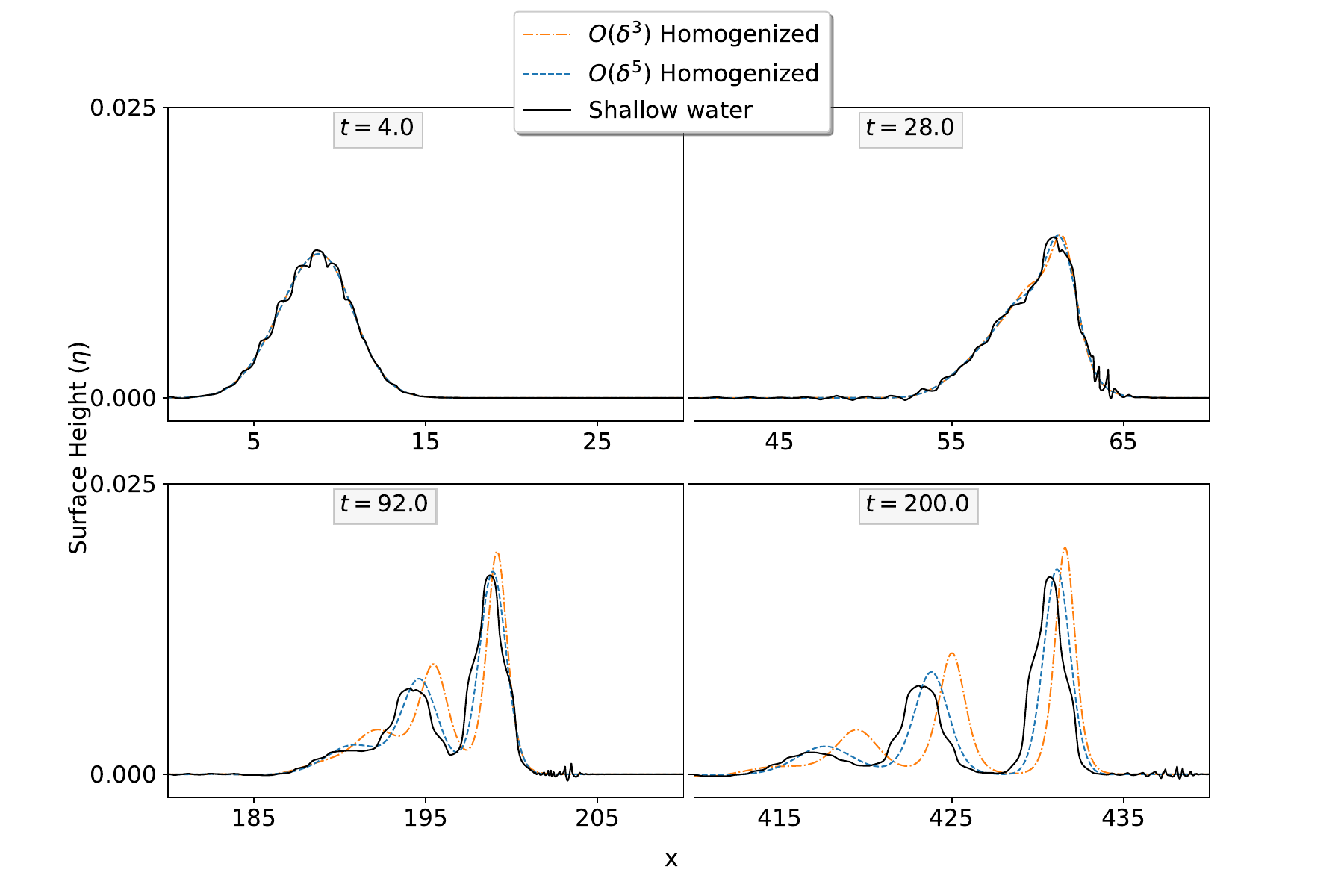}
    \caption{Comparison of homogenized and direct solutions, with smooth
        bathymetry \eqref{scenario_b}.
        The surface elevation $\eta = h+b$ is shown.\label{fig:scenario_b}}
\end{figure}

Finally, we briefly examine the error in the homogenized solution surface height as a function of time.  
In order to compare the direct solution with the
homogenized solution, we  compute the sliding average of the solution on a space window of width $\delta$, defined as
\[
    \bar{\eta}(x,t) = \frac{1}{\delta}\int_{-\delta/2}^{\delta/2}\eta(x+\xi,t)\,d\xi.
\]
We denote by $\eta_{\rm FV}$ the approximation of $\bar{\eta}(x,t)$ computed by the finite volume method.
In Figure \ref{fig:w_vs_time}, we plot a Wasserstein-1 distance between the homogenized (PS) solution and the sliding average of the direct finite volume (FV) solution:
\begin{align}
    W_1(\eta_\textup{PS}, 
    \eta_\textup{FV}) = \frac{\Delta x}{\sum_j
    (\eta_\textup{FV})_j}\sum_j |
    (H_\textup{PS})_j -
    (H_\textup{FV})_j|,
    \label{eq:W1}
\end{align}
where we denoted by $H_i$ the discrete primitive of $\eta$, i.e.\
$H_i = \Delta x \sum_{j\leq i} \eta_j $, and $\Delta x$ is the grid spacing used for the pseudospectral solver. 
 We used $\Delta x = 1/8$, which is much larger than the the grid spacing $\Delta x_\textup{FV}$ employed by the FV scheme (we used $\Delta x_\textup{FV} = 1/640$) but still sufficiently small to obtain a very accurate solution, thanks to the spectral accuracy of the method.
For a justification of  formula
\eqref{eq:W1}, see Appendix \ref{sec:Wasserstein}.

The choice of the Wasserstein-1 distance appears quite natural to measure the discrepancy between the two profiles, since it provides a measure of the effective displacement of the two profiles.
We see that for long times the difference grows roughly linearly, with a rate that becomes smaller if higher-order terms are included.  Here we have also included the pseudospectral solution of the homogenized equations including all terms up to order $\delta^5$ and just the linear terms of order $\delta^7$, as discussed above.
At the final time $t=200$, even for the least accurate
approximation we have a phase error of about one meter after the waves have propagated a distance of more than 400 meters.


The discrepancy between the curves at very early times seems to be due to the error in the approximation of the integral in the computation of primitive function for the pseudospectral solution. We used the midpoint rule, which introduces an error of $O(\Delta x^2)$.

\begin{figure}
    \center
   \includegraphics[width=0.8\textwidth]{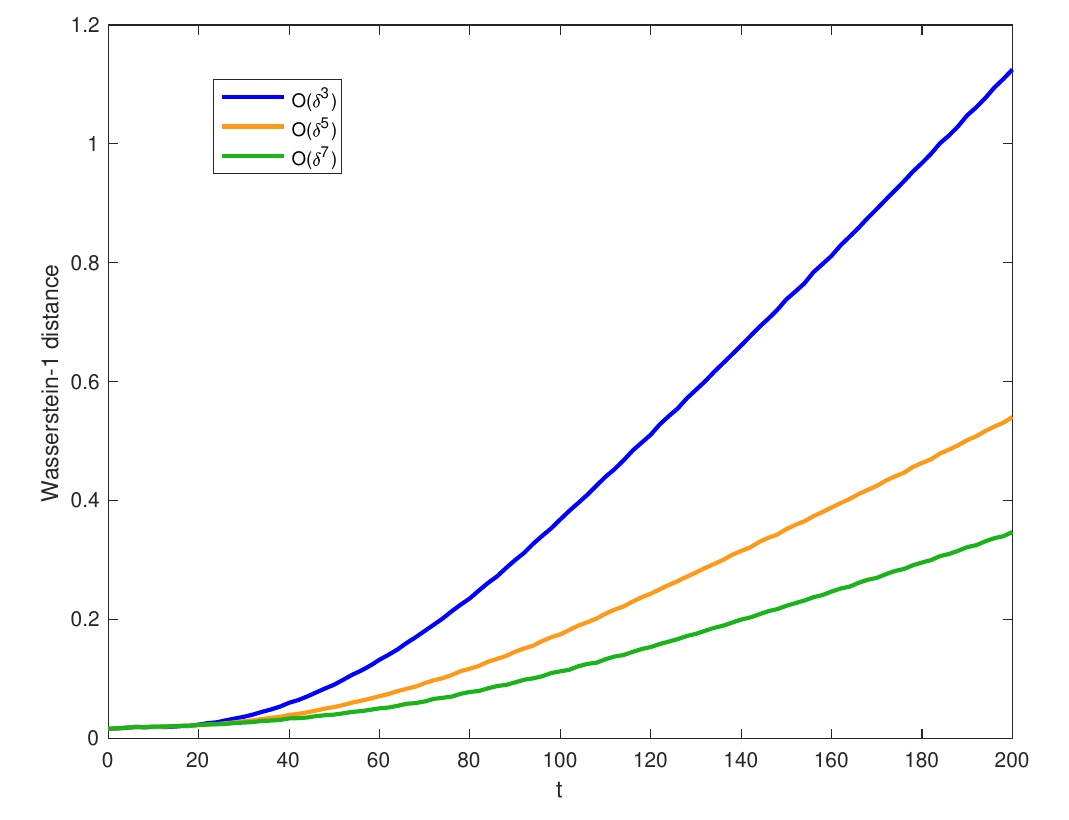}
    \caption{Wasserstein distance between the sliding average of the numerical solution of the original shallow water system obtained by the finite volume method, and the solutions of the  homogenized at various orders, obtained by the pseudospectral method. \label{fig:w_vs_time}}
\end{figure}

\section{Conclusions and future work}

In this work we have shown that water waves over periodic bathymetry can be accurately
described by an effective system of constant-coefficient equations derived via multiple-scale
perturbation theory.  The resulting equations are dispersive and possess periodic and solitary 
traveling-wave solutions.  This is in contrast to typical solutions of the shallow water
equations, which generally exhibit wave breaking.  Nevertheless, numerical solutions of
the shallow water equations in the presence of periodic bathymetry  
are in close agreement with solutions of the effective medium equations.
Numerical solution of the effective equations can be drastically more efficient; using
state-of-the-art discretizations for both the original shallow water equations and the
effective homogenized equations, we have seen a speedup of more than two orders of magnitude
by using the latter.

Although we have focused on scenarios that give rise to solitary waves, it should be emphasized
that a number of other solution behaviors exist.  In addition to that already-mentioned periodic
traveling waves, for large initial data and/or small bathymetry variation, shock formation is observed; 
initial data consisting of a negative perturbation to a flat surface yield still other kinds
of behavior that have yet to be explored.  Even the solitary and periodic 
wave solutions are not true traveling waves as they are modulated by the periodic bathymetry; in this
sense they resemble breathers.  Further investigation of the structure of all of these types
of solutions is an area for future work.

This study is similar in many ways to that of LeVeque \& Yong \cite{leveque2003}, which was
focused on the $p$-system.  Indeed, the qualitative solution behaviors and the structure of 
the equations derived have much in common, even though periodicity comes into the equations
in somewhat different ways.  The Wolfram \textit{Mathematica} code developed in the present work can be adapted
in a very straightforward way to reproduce the results of \cite{leveque2003}, and such a reproduction
is also included with the code written for this paper.  We hope that this may facilitate
similar analysis of other systems.

Finally, this work could be extended in important mathematical and physical directions, for instance
by quantifying the space and time scales of validity of the various approximations, or
starting from a more accurate water wave model that already includes dispersion.

\appendix
\section{Appendix}\label{sec:appendix}

Here, we recall the averaging functionals and operators  introduced in \cite{yong2002}, and prove some of their properties that we have used in our work.

In what follows, let $f,g:[0,1]\to\mathbb{R}$ and $\varphi,\psi:\mathbb{R}\to\mathbb{R}$
  denote some sufficiently smooth functions---continuous or continuously differentiable, guaranteeing the existence of the definite and indefinite integrals appearing below. We always assume that the compositions  
$\varphi\circ f$ and $\psi\circ f$ are defined on the whole interval $[0,1]$.

\subsection{The functional $\mean{\ \cdot\ }$ }\label{sec:averageop}

The integral average of $f$ over the interval $[0,1]$, denoted by $\mean{f}\in\mathbb{R}$, is defined as
$$
    \mean{f} := \int_0^1 f.
$$
For functions satisfying $f(0)=f(1)$, we have
\begin{equation}\label{FTC}
   \mean{(\varphi\circ f)\cdot f'}=0,
\end{equation}
since $\int_0^1 ((\varphi\circ f)\cdot f')=\Phi(f(1))-\Phi(f(0))$, with $\Phi$ denoting an antiderivative of $\varphi$.

\subsection{The operators $\{\,\cdot\, \}$ and  $\fluctint{\,\cdot\,}$ }\label{sec:fluctintop}

The fluctuating part of the function $f$, denoted by $\{ f \}:[0,1]\to\mathbb{R}$, is defined as 
$$\{f\}:=f-\mean{f}.$$
Clearly, we have $\mean{\{f\}}=0$.

In \cite{yong2002}, the integral of the fluctuating part, denoted by 
$\llbracket f \rrbracket:[0,1]\to\mathbb{R}$, is defined for any $y\in [0,1]$ as
$$
   \llbracket f \rrbracket(y):=\int_s^y \{f\}, \ \ \text{where $s$ is chosen so that } \mean{\fluctint{f}} = 0.
$$
(Note the typo in \cite[Appendix A]{yong2002}: in their corresponding sentence ``where $s$ is chosen such that\ldots'' the symbol $\{\}$ should be replaced by $\llbracket\ \rrbracket$.)  In the above---somewhat implicit---definition of $\fluctint{f}$, one can think of $s$ as choosing  the constant of integration suitably:  $\llbracket f \rrbracket$ is the antiderivative of $\{f\}$ with zero mean. We can easily check that this operator can be rewritten more explicitly as our formula \eqref{[[]]expldef} in Section \ref{sec:averaging}.

\begin{prop}
We also have the representations
\begin{equation}\label{[[]]doubleint} \fluctint{f} (y)= \int_0^y f(\xi)d\xi  +\left(\frac{1}{2}-y\right) \int_0^1f(\xi)d\xi-\int_0^1 \int_0^\tau f(\xi)d\xi d\tau
\end{equation}
and
\begin{equation}\label{[[]]singleint} \fluctint{f} (y)= \int_0^y f(\xi)d\xi  - \int_0^1 \left(\frac{1}{2} +y - \xi\right)f(\xi)d\xi.
\end{equation}
\end{prop}
\begin{proof}
By using \eqref{[[]]expldef}, we get
\[\fluctint{f} (y)= \int_0^y (f(\xi)-\mean{f})d\xi  -\int_0^1 \int_0^\tau (f(\xi)-\mean{f})d\xi d\tau=\]
\[\int_0^y f(\xi)d\xi -\mean{f} y -\int_0^1 \int_0^\tau f(\xi)d\xi d\tau
+\int_0^1  \mean{f}\tau d\tau,
\]
which is \eqref{[[]]doubleint}.
To prove \eqref{[[]]singleint}, let $F:[0,1]\to\mathbb{R}$ denote the antiderivative $F(\tau):=\int_0^\tau f$. Integration by parts then yields
\[\int F(\tau)d\tau=\int 1\cdot F(\tau)d\tau=\tau F(\tau)-\int\tau f(\tau)d\tau,\]
hence
\[\int_0^1 \int_0^\tau f(\xi)d\xi d\tau=\int_0^1 F(\tau)d\tau=1\cdot F(1)-0\cdot F(0)-\int_0^1 \xi f(\xi)d\xi.\]
Plugging this into \eqref{[[]]doubleint}, we obtain \eqref{[[]]singleint}.
\end{proof}
Next, we formulate some results implying that  certain integral averages involving $\fluctint{f}$ vanish (cf.~\eqref{FTC}). One such statement is found in \cite{yong2002}, where it is proved that 
\begin{align} \label{fbf0}
    \mean{\,f\fluctint{f}\,}=0.
\end{align}
From this comes the identity 
\begin{equation}\label{integrationbypartsfor[[]]}
    \mean{\,f\fluctint{g}\,}=-\mean{\, \fluctint{f}g\, },
\end{equation}
which is, essentially, a form of integration by parts.

\begin{prop}
Assume that $f(0)=f(1)$. Then for any 
continuous functions $\varphi$ and $\psi$  we have
\begin{equation}\label{phipsiformula}
    \mean{\, (\varphi\circ f)\cdot f'\cdot \  \fluctint{\,(\psi\circ f)\cdot f'\,}\,}=0.
\end{equation}
\end{prop}
\begin{proof}
Let $\Psi$ denote an antiderivative of $\psi$. 
Then $\Psi\circ f$ is an antiderivative of $(\psi\circ f)\cdot f'$, so \eqref{[[]]doubleint} yields that
\[
\fluctint{\,(\psi\circ f)\cdot f'\,}(y)=\Psi(f(y))-\Psi(f(0))+\left(\frac{1}{2}-y\right)(\Psi(f(1))-\Psi(f(0)))-\text{const}_1=
\]
\[
\Psi(f(y))-\text{const}_2,
\]
where we took into account the fact that $f(0)=f(1)$, and 
the double integral in \eqref{[[]]doubleint} and $\Psi(f(0))$ have been merged into the constants. 
Therefore 
\[
\mean{\,(\varphi\circ f)\cdot f'\cdot \  \fluctint{\,(\psi\circ f)\cdot f'\,}\,}=\mean{\, (\varphi\circ f)\cdot \big(\Psi\circ f-\text{const}_2\big)\cdot f'\,},
\]
so \eqref{FTC} becomes applicable, proving \eqref{phipsiformula}.
\end{proof}
\begin{cor} Assume that $f>0$ and $f(0)=f(1)$. Then for any $\alpha, \beta\in\mathbb{R}$ we have
\[
    \mean{\,f^\alpha f'\cdot \  \fluctint{\,f^\beta f'\,}\,}=0.
\]
\end{cor}

Another generalization of \eqref{fbf0} is as follows.  Let $\fluctint{\cdot}_j$ denote the
operator $\fluctint{\cdot}$ applied  $j\in\mathbb{N}^+$ times; e.g.,~$\fluctint{f}_3=\fluctint{\fluctint{\fluctint{f}}}$.
\begin{prop}
    If $j=2k+1$, $k=0,1,\ldots$ (i.e., if $j$ is a positive odd number), then
    $$\avg{f\fluctint{f}_j} = 0.$$
\end{prop}
\begin{proof}
    By repeatedly applying \eqref{integrationbypartsfor[[]]}, we have
    $$\avg{f\fluctint{f}_j} = (-1)^k\avg{\fluctint{f}_{k}\fluctint{f}_{k+1}} = (-1)^{k+1}\avg{\fluctint{f}_{k+1}\fluctint{f}_{k}}.$$
    But the last equality implies that all these quantities vanish.
\end{proof}

Expressions such as $\mean{f^2\fluctint{f}}$ 
do not vanish in general, as shown, for example,
by the function
$$f(y):=\sin (2 \pi  y)+\cos (2 \pi  y)+\cos (4 \pi  y).$$
However, under a certain symmetry assumption, $\mean{f^2\fluctint{f}}=0$. This is the content of Proposition \ref{symmprop} below. 

We say that a function $g:[0,1]\to\mathbb{R}$ is \textit{even with respect to the midpoint} if 
\begin{equation}\label{def:evenwrtmidpoint}
    \forall y\in[0,1]: \quad g(y)=g(1-y),
\end{equation}
and \textit{odd with respect to the midpoint} if 
\begin{equation}
\label{def:oddwrtmidpoint}
\forall y\in[0,1]: \quad g(y)=-g(1-y).
\end{equation}
\begin{lem}\label{oddlemma}
    Suppose that $f$ is even with respect to the midpoint. Then $\fluctint{f}$ is odd with respect to the midpoint.
\end{lem}
\begin{proof}
By using \eqref{[[]]singleint}, the substitution $\xi=1-z$, and the fact that $f$ is even with respect to the midpoint, we have   
\[
\fluctint{f}(1-y)=\int_0^{1-y} f(\xi)d\xi  - \int_0^1 \left(\frac{1}{2} +(1-y) - \xi\right)f(\xi)d\xi=
\]
\[
=-\int_1^{y} f(1-z)d z  + \int_1^0 \left(\frac{1}{2} +(1-y) - (1-z)\right)f(1-z)d z=
\]
\[
=\int_y^{1} f(1-z)d z  - \int_0^1 \left(\frac{1}{2} +z-y\right)f(1-z)d z=\int_y^{1} f(z)d z  - \int_0^1 \left(\frac{1}{2} +z-y\right)f(z)d z.
\]
Therefore, our claim $\fluctint{f}(y)=-\fluctint{f}(1-y)$ is equivalent to
\[
\int_0^y f(\xi)d\xi  - \int_0^1 \left(\frac{1}{2} +y - \xi\right)f(\xi)d\xi=-\int_y^{1} f(\xi)d \xi  + \int_0^1 \left(\frac{1}{2} +\xi-y\right)f(\xi)d \xi,
\]
that is, to
\[
\int_0^y f(\xi)d\xi +\int_y^{1} f(\xi)d \xi = \int_0^1 \left(\frac{1}{2} +y - \xi\right)f(\xi)d\xi + \int_0^1 \left(\frac{1}{2} +\xi-y\right)f(\xi)d \xi,
\]
which is clear, since both sides are equal to $\mean{f}$.
\end{proof}
\begin{prop}\label{symmprop}
    Suppose that $f$ is even with respect to the midpoint. Then for any continuous function $\varphi$  we have
$$\mean{\,(\varphi\circ f)\cdot \fluctint{f}\,}=0.$$
\end{prop}
\begin{proof}
By splitting the integral and using the substitution
$y=1-z$, we can use Lemma \ref{oddlemma} in the second component and get
    \[\mean{(\varphi\circ f)\cdot \fluctint{f}}=
    \int_0^\frac{1}{2}\Big((\varphi\circ f)\cdot \fluctint{f}\Big)+
    \int_{\frac{1}{2}}^1 \Big((\varphi\circ f)\cdot \fluctint{f}\Big)   =
    \]
\[
\int_0^\frac{1}{2}\Big((\varphi\circ f)\cdot \fluctint{f}\Big)-\int_\frac{1}{2}^0 \varphi( f(1-z))\cdot \fluctint{f}(1-z)d z
    =    
\]
\[
\int_0^\frac{1}{2}\Big((\varphi\circ f)\cdot \fluctint{f}\Big)+\int_\frac{1}{2}^0 \varphi( f(z))\cdot \fluctint{f}(z)d z
    =0.    
\]
\end{proof}

\subsection{Periodic functions}
We can extend the results of Sections \ref{sec:averageop}--\ref{sec:fluctintop} to a larger class of functions as follows.  Here, we consider---sufficiently smooth---periodic functions $f:\mathbb{R}\to\mathbb{R}$ with (without loss of generality) period 1:
$$
    f(y+1) = f(y) = f(y-\lfloor y \rfloor)\quad\quad(y\in\mathbb{R}).
$$
For some $\sigma\in (0,1)$, let $f_\sigma$ denote the shifted function
$$
    f_\sigma(y) := f(y-\sigma)\quad\quad(y\in\mathbb{R}).
$$
For periodic functions, let $\mean{f}$ denote the integral average over a single period ($\mean{f} := \int_0^1 f$), and we define $\fluctint{f}$ again by \eqref{[[]]expldef} but this time for any $y\in\mathbb{R}$.

It is well-known that a shift does not alter the single-period average,
\begin{equation}\label{shiftdoesnotchange}
    \mean{f_\sigma}=\mean{f},
\end{equation}
since
\[
 \int_0^1 f(y-\sigma) d y= \int_{-\sigma}^{1-\sigma} f=\left(\int_{-\sigma}^{0}+\int_{0}^{1}+\int_{1}^{1-\sigma}\right) f=
    \left(\int_{-\sigma}^{0}+\int_{0}^{1}+\int_{0}^{-\sigma}\right) f=
    \int_0^1 f.
\]

To investigate the corresponding property for the operator $\fluctint{\cdot}$, note first that an antiderivative of a 1-periodic function is not necessarily 1-periodic. 
\begin{lem}\label{antiderperlemma}
Suppose that the function $g:\mathbb{R}\to\mathbb{R}$ is 1-periodic, and let $G(y):=\int_0^y g$. Then $G$ is 1-periodic if and only if $\mean{g}$=0.
\end{lem}
\begin{proof}
\[
G(y+1)=\int_0^{y+1} g=\int_0^{1} g+\int_1^{y+1} g=
\mean{g}+\int_0^{y} g(\xi+1)d \xi=
\mean{g}+\int_0^{y} g=\mean{g}+G(y).
\]
\end{proof}
\begin{lem}\label{lem:perlemma}
Suppose that $f:\mathbb{R}\to\mathbb{R}$ is 1-periodic. Then $\fluctint{f}$ is also 1-periodic.
\end{lem}
\begin{proof}
    The first term on the right-hand side of \eqref{[[]]expldef} is 1-periodic in $y$ due to Lemma \ref{antiderperlemma} with the periodic function $g:=\{f\}$, since $\mean{g}=\mean{\{f\}}=0$.
\end{proof}
\begin{lem}\label{[[]]periodic}
Suppose that $f:\mathbb{R}\to\mathbb{R}$ is 1-periodic. Then the shift operator commutes with the operator $\fluctint{\cdot}$: 
\[
\fluctint{f_\sigma}  = \fluctint{f}_\sigma.
\]
\end{lem}
\begin{proof} Let $F(y):=\int_0^y \{f\}$. Then 
$F(1)=0$, so due to Lemma \ref{antiderperlemma}, $F$ is 1-periodic. Then, by \eqref{[[]]expldef}, 
the claim of the lemma is equivalent to any of the following:
\[
\int_0^y \{f\}(\xi -\sigma ) \, d\xi -\int _0^1\int _0^{\tau }\{f\}(\xi
   -\sigma )d\xi d\tau =\int_0^{y-\sigma } \{f\}(\xi ) \, d\xi -\int
   _0^1\int _0^{\tau }\{f\}(\xi )d\xi d\tau,
\]
\[
\int_{-\sigma }^{y-\sigma } \{f\}(\xi ) \, d\xi -\int _0^1\int
   _{-\sigma }^{\tau -\sigma }\{f\}(\xi )d\xi d\tau =\int_0^{y-\sigma }
   \{f\}(\xi ) \, d\xi -\int _0^1\int _0^{\tau }\{f\}(\xi )d\xi d\tau,
\]
\[
F(y-\sigma )-F(-\sigma)-\int_0^1 (F(\tau -\sigma )-F(-\sigma )) \, d\tau =F(y-\sigma )-F(0)-\int_0^1 (F(\tau )-F(0)) \, d\tau, 
\]
\[
\int_0^1 F(\tau -\sigma ) \, d\tau =\int_0^1 F(\tau ) \, d\tau,
\]
which is just \eqref{shiftdoesnotchange} applied to $F$.
\end{proof}

We say that a 1-periodic function $f$ is \emph{translation-even} if there exists a shift
$\sigma\in (0,1)$ such that $f_\sigma$ is even with respect to the midpoint of the interval $[0,1]$ (see \eqref{def:evenwrtmidpoint}).  
\begin{prop}\label{translationevensymmprop}
    Suppose that $f$ is a translation-even 1-periodic function with shift $\sigma$. Then for any continuous function $\varphi$  we have
$$\mean{(\varphi\circ f)\cdot \fluctint{f}}=0.$$
\end{prop}
\begin{proof}
Due to Lemma \ref{lem:perlemma}, $(\varphi\circ f)\cdot \fluctint{f}$ is 1-periodic, so \eqref{shiftdoesnotchange} applied to this function and Lemma \ref{[[]]periodic} yield that
\[
\mean{(\varphi\circ f)\cdot \fluctint{f}}=
\mean{\left((\varphi\circ f)\cdot \fluctint{f}\right)_\sigma}=\mean{(\varphi\circ f)_\sigma\cdot \fluctint{f}_\sigma}=
\mean{(\varphi\circ f)_\sigma\cdot \fluctint{f_\sigma}}=\mean{(\varphi\circ f_\sigma)\cdot \fluctint{f_\sigma}}, 
\]
therefore Proposition \ref{symmprop} applied to $f_\sigma$ completes the proof.
\end{proof}

Analogously, we say that a 1-periodic function $f$ is \emph{translation-odd} if there exists a shift
$\sigma\in (0,1)$ such that $f_\sigma$ is odd with respect to the midpoint of the interval $[0,1]$ (see \eqref{def:oddwrtmidpoint}).  

\begin{prop}\label{translationodddoublebracket}
    Suppose that $f$ is a translation-even 1-periodic function. Then $\fluctint{f}$ is translation-odd.
\end{prop}
\begin{proof} We know that $\exists\  \sigma\in(0,1)$ such that  $f_\sigma$ is even with respect to the midpoint.
So, according to Lemma \ref{oddlemma}, $\fluctint{f_\sigma}$ is odd 
with respect to the midpoint. But this means, by using 
Lemma \ref{[[]]periodic}, that  $\fluctint{f}$ is translation-odd.
\end{proof}

\subsection{Examples}

When simplifying the coefficients in Section \ref{sec:homog}, we have applied several identities which follow from the above claims; below we highlight some of them. 

The first group of identities, Proposition \ref{Propfirstgroup}, is important also in eliminating the derivative $H'$ so that one can extend the results to non-smooth bathymetries $H$. 
\begin{prop}\label{Propfirstgroup} Suppose that the 1-periodic function $H$ is positive. Then
\begin{itemize}
    \item $\mean{H^{-5}\cdot (\fluctint{H^{-1}})^2\cdot H'}=\frac{1}{2}\mean{H^{-5}\cdot  \fluctint{H^{-1}}}-\frac{1}{2}\avg{H^{-1}}\mean{H^{-4}\cdot  \fluctint{H^{-1}}}$
    
    \item $\mean{H^{-3}\cdot (\fluctint{H^{-2}})^2\cdot H'}=\mean{H^{-4}\cdot\fluctint{H^{-2}}}$

    \item $\mean{H^{-1}\cdot\fluctint{\fluctint{H^{-3}\cdot\fluctint{\fluctint{H^{-1}}}\cdot  H'}}}=\mean{H^{-2}\cdot\fluctint{H^{-1}}\cdot\fluctint{\fluctint{H^{-1}}}}$

    \item $\mean{H^{-3}\cdot\fluctint{H^{-2}}\cdot\fluctint{\fluctint{H^{-1}}}\cdot H'}=\frac{1}{2}\mean{H^{-2}\cdot\fluctint{H^{-1}}\fluctint{H^{-2}}}-\frac{1}{2}\mean{\fluctint{H^{-1}}\fluctint{H^{-4}}}+\frac{1}{2}\mean{H^{-2}}\mean{\fluctint{H^{-1}}\fluctint{H^{-2}}}$

    \item $\mean{H^{-1}\cdot\fluctint{\fluctint{H^{-4}\cdot\fluctint{H^{-1}}\cdot H'}}}=\frac{1}{3}\mean{H^{-3}\cdot(\fluctint{H^{-1}})^2}-\frac{1}{3}\mean{\fluctint{H^{-1}}\fluctint{H^{-4}}}+\frac{1}{3}\mean{H^{-1}}\mean{\fluctint{H^{-1}}\fluctint{H^{-3}}}$

 \item $\mean{H^{-4}\cdot\fluctint{H^{-2}}\cdot\fluctint{H^{-1}}\cdot H'}=\\ \frac{1}{3}\mean{H^{-4}\cdot\fluctint{H^{-2}}}-\frac{1}{3}\mean{H^{-1}}\mean{H^{-3}\cdot\fluctint{H^{-2}}}+\frac{1}{3}\mean{H^{-5}\cdot\fluctint{H^{-1}}}- \frac{1}{3}\mean{H^{-2}}\mean{H^{-3}\cdot\fluctint{H^{-1}}}$
    
\end{itemize}
\end{prop}
\begin{proof}
    We prove only the first identity; the others follow similarly. Let us consider its left-hand side and denote the corresponding indefinite integral by $I$:
\[
I:=\int \left(H^{-5}\cdot (\fluctint{H^{-1}})^2 \cdot H'\right).
\]
    By using integration by parts, we obtain
\begin{equation}\label{IBP}
I=H^{-5}\cdot (\fluctint{H^{-1}})^2 \cdot H-\int \left(\left(H^{-5}\cdot (\fluctint{H^{-1}})^2\right)' \cdot H\right).
\end{equation}
Inside the integral on the right, by using the product rule, we have
\[
\left(H^{-5}\cdot (\fluctint{H^{-1}})^2\right)'\cdot H=2H^{-4}\cdot\fluctint{H^{-1}}\cdot\fluctint{H^{-1}}'-5H^{-5}(\fluctint{H^{-1}})^2\cdot H'.
\]
Notice that the representation \eqref{[[]]doubleint} and the fundamental theorem of calculus show that $\fluctint{f}'=f-\avg{f}$, hence
$\fluctint{H^{-1}}'=H^{-1}-\avg{H^{-1}}$. Plugging these into \eqref{IBP} yields
\[
I=H^{-4}\cdot (\fluctint{H^{-1}})^2-2\int \left(H^{-5}\cdot\fluctint{H^{-1}}\right)+2\avg{H^{-1}}\int \left(H^{-4}\cdot\fluctint{H^{-1}}\right)+5 I,
\]
from which the unknown integral $I$ can easily be expressed as
\[
I=-\frac{1}{4}H^{-4}\cdot (\fluctint{H^{-1}})^2+\frac{1}{2}\int \left(H^{-5}\cdot\fluctint{H^{-1}}\right)-\frac{1}{2}\avg{H^{-1}}\int \left(H^{-4}\cdot\fluctint{H^{-1}}\right).
\]
Now, when we switch to definite integrals $\int_0^1$, the first term on the right, $-\frac{1}{4}H^{-4}\cdot (\fluctint{H^{-1}})^2$, disappears, since $H$ is 1-periodic, and, due to Lemma \ref{lem:perlemma}, $\fluctint{H^{-1}}$ is also 1-periodic. In this way, we recover the desired right-hand side.
\end{proof}
\begin{rem}
     It would be interesting to find some general rules that include the different examples above as special cases.
\end{rem}
The second group of identities shows some conversion rules which can be repeatedly applied until we find an expression that appears on the left-hand side of one of the formulae in Proposition \ref{Propfirstgroup}.
\begin{rem}
   Repeated application of these identities in the code is achieved by invoking Wolfram \textit{Mathematica}'s powerful \emph{{\texttt{ReplaceRepeated}}} command.
\end{rem}
\begin{prop} Suppose that the 1-periodic function $H$ is positive. Then
\begin{itemize}
\item $\mean{H^{-3}\cdot(\fluctint{\fluctint{H^{-1}}})^2\cdot H'}=\mean{H^{-1}\cdot\fluctint{\fluctint{H^{-3}\cdot\fluctint{\fluctint{H^{-1}}}\cdot H' }})}$ 

\item $\mean{H^{-1}\cdot\fluctint{\fluctint{H^{-3}\cdot\fluctint{H^{-2}}\cdot H' }})}=-\mean{\fluctint{H^{-1}}\cdot \fluctint{H^{-3}\cdot \fluctint{H^{-2}} \cdot H'}}=\mean{H^{-3}\cdot\fluctint{H^{-2}}\cdot\fluctint{\fluctint{H^{-1}}}\cdot H'}$

\item $\mean{H^{-2}\cdot\fluctint{H^{-4}\cdot\fluctint{H^{-1}} \cdot H'}}=-\mean{H^{-4}\cdot\fluctint{H^{-2}} \cdot\fluctint{H^{-1}} \cdot H'}$

\end{itemize}
\end{prop}
\begin{proof}
We prove only the first identity; the other proofs are clearly analogous. We regroup and apply  \eqref{integrationbypartsfor[[]]} twice to get 
\[
\mean{H^{-3}\cdot(\fluctint{\fluctint{H^{-1}}})^2\cdot H'}=\mean{\fluctint{\fluctint{H^{-1}}}\cdot H^{-3}\cdot \fluctint{\fluctint{H^{-1}}}\cdot H'}=
\]
\[
-\mean{\fluctint{H^{-1}}\cdot \fluctint{H^{-3}\cdot \fluctint{\fluctint{H^{-1}}}\cdot H'}}=
\mean{H^{-1}\cdot \fluctint{\fluctint{H^{-3}\cdot \fluctint{\fluctint{H^{-1}}}\cdot H'}}}.
\]
\end{proof}

Finally, under a certain symmetry assumption, many expressions will vanish according to the following.
\begin{prop} Suppose that the 1-periodic function $H$ is positive and translation-even. Then
\begin{itemize}
    \item $\mean{\fluctint{H^{-1}}\cdot\fluctint{\fluctint{H^{-2}}}}=0$
    \item $\mean{H^{-1}\cdot\fluctint{H^{-k}}}=0$ \quad for any $k\in\mathbb{N}^+$.
\end{itemize}
\end{prop}
\begin{proof}
   We use Proposition \ref{translationodddoublebracket} and the ``translated version'' of the following elementary fact (cf.~the proof of Proposition \ref{symmprop}): the definite integral of the product of an even and an odd function over an interval symmetric about the origin vanishes.
\end{proof}

\section{Coefficients of the homogenized equations} \label{sec:coefficients}
In this section we provide the coefficients and some other details of the
high-order homogenized equations that are too lengthy for main text.
The coefficients of \eqref{o4avg-xxt} are given by
\begin{subequations}
\begin{align}
    \mu & = \frac{\mean{\fluctint{H^{-1}}^2}}{\mean{H^{-1}}^2} \label{mu} \\
    \gamma & = \frac{\avg{\fluctint{H^{-1}}\fluctint{H^{-2}}}}{\Hm{-1}^2} \label{gamma} \\
    \nu_1 & = \frac{\avg{H^{-1}(\fluctint{\fluctint{H^{-1}}})^2}}{\Hm{-1}^3} \\
    \nu_2 & = 3 \frac{\avg{(\fluctint{\fluctint{H^{-1}}})^2}}{\Hm{-1}^2} \\
    \alpha_1 & = \frac{2}{\mean{H^{-1}}^2} \left( \mean{H^{-2}}^2 - 2 \mean{H^{-3}} \ \mean{H^{-1}} \right) \label{eq:alpha1}\\
    \alpha_2 & = \frac{3\mean{H^{-2}}^2 - 2 \mean{H^{-1}} \ \mean{H^{-3}} - 3\mean{H^{-4}}}{2\Hm{-1}^2} \label{eq:alpha2}\\
    \alpha_3 & =\frac{1}{\mean{H^{-1}}^3}\left( \mean{H^{-2}}^2 - \mean{H^{-3}} \ \mean{H^{-1}} \right) \label{eq:alpha3}\\
    \alpha_4 & = \frac{3\Hm{-2}^3  - 4\Hm{-1}\Hm{-2}\Hm{-3} - 3 \Hm{-2}\Hm{-4} + 4\Hm{-1}\Hm{-5}}{\Hm{-1}^2} \\
    \alpha_5 & =  \frac{2\Hm{-2}^3 - 6 \Hm{-1}\Hm{-2}\Hm{-3} + 6\Hm{-1}^2\Hm{-4}}{\Hm{-1}^3}  \\
    \alpha_6 & =  \frac{ 3\Hm{-2}^3 - 7 \Hm{-1}\Hm{-2}\Hm{-3} + 3 \Hm{-1}^2\Hm{-4} - 3\Hm{-2}\Hm{-4} + 6\Hm{-1}\Hm{-5} }{\Hm{-1}^3}  \\ \\
    \alpha_7 & = \frac{\Hm{-2}^3 - 2 \Hm{-1}\Hm{-2}\Hm{-3} + \Hm{-1}^2\Hm{-4}}{\Hm{-1}^4}   \\
    \alpha_8    & =  2\left(\mu \frac{\Hm{-2}}{\Hm{-1}} - \gamma\right) \\
    \alpha_9    & =  \mu \frac{\Hm{-2}}{\Hm{-1}}.
\end{align}
\end{subequations}
In order to simplify the expression of the 5th-order nonlinear terms, we consider
the special case of translation-even periodic functions $H(y)$, which include both
of the examples given in the paper (piecewise-constant and sinusoidal bathymetry).
Then the nonlinear 5th-order terms are given by
\begin{align*}
    F(\heta,\hq) & = \beta_1 \hq^4 \heta_x + \beta_2 \heta^4 \heta_x + \beta_3 \heta^2 \ \hq^2 \heta_x + \beta_4 \heta \ \hq^3 \hq_x + \beta_5 (\heta_x)^3 + \beta_6 \heta \ \heta_x \heta_{xx} + \beta_7 \heta^2 \heta_{xxx} \\
     &  + \beta_8 \heta_x \hq \ \hq_{xx} + \beta_9 \hq \ \heta^3 \hq_x + \beta_{10} \heta_{xx} \hq \ \hq_x + \beta_{11} \heta_x (\hq_x)^2 + \beta_{12} \hq^2 \heta_{xxx} + \beta_{13} \heta \ \hq_x \hq_{xx} + \beta_{14} \heta \ \hq \ \hq_{xxx} 
\end{align*}
where
\begin{align*}
    \beta_1 & = \frac{1}{c^2}\left( \theta_3^2-\frac{21}{4}\theta_2^2 \theta_3 + \frac32\theta_2\theta_4 + \frac32\theta_3\hat{\theta}_4+\frac{15}{2}\theta_2\hat{\theta}_5 - \frac52\hat{\theta}_6-\frac{15}{4}\theta_7\Hm{-1}^{-2} + \frac{9}{4}\left(\theta_2^2-\hat{\theta}_4 \right)^2 \right) \\
    \beta_2 & = c^2\left( \theta_2^4-3\theta_2^2\theta_3+\theta_3+2\theta_2\theta_4-\theta_5 \right) \\
    \beta_3 & = - 6\theta_5 -15\hat{\theta}_6 + \frac{9}{2}\theta_2^4-16\theta_2^2\theta_3 + 7\theta_3^2 + 12\theta_2\theta_4-\frac{9}{2}\theta_2^2\hat{\theta}_4+3\theta_3\hat{\theta}_4+12\theta_2\hat{\theta}_5 \\
    \beta_4 & = \frac{1}{c^2}\left( -20\hat{\theta}_6 + 6\theta_2^4-22\theta_2^2\theta_3+8\theta_3^2+12\theta_2\theta_4-6\theta_2^2\hat{\theta}_4+6\theta_3\hat{\theta}_4+16\theta_2\hat{\theta}_5 \right) \\
    \beta_5 & = c^2\left( -2\zeta_{13} + \zeta_{122} + 2 \zeta_{212} + \zeta_{311} + 3\zeta_{14} - 3\gamma\theta_2 - \zeta_{22} + 8\mu\theta_2^2-2\mu\theta_3-3\mu\hat{\theta}_4 \right) \\
    \beta_6 & = c^2\left(-16 \gamma \theta_2+26 \mu\theta_2^2-10 \mu\theta_3 \right) \\
    \beta_7 & = c^2 \left(2\zeta_{13} + \zeta_{22} -6\gamma\theta_2+5\mu\theta_2^2-2\mu\theta_3 \right) \\
    \beta_8 & = 4\zeta_{122} + 8\zeta_{212} + 4\zeta_{311} + 12\zeta_{14} - 12\gamma\theta_2 - 2 \zeta_{22}-4\zeta_{13}+27\mu\theta_2^2-6\mu\theta_3-9\mu\hat{\theta}_4 \\
    \beta_9 & = 2\theta_2^4-8\theta_2^2\theta_3+4\theta_3^2+8\theta_2\theta_4-8\theta_5 \\
    \beta_{10} & = -4\zeta_{13} - 2\zeta_{22} - 8\gamma\theta_2+28\mu\theta_2^2-12\mu\theta_3 \\
    \beta_{11} & = 2\zeta_{13} + \zeta_{22} - 12\gamma \theta_2 + 22\mu\theta_2^2 - 10\mu\theta_3 \\
    \beta_{12} & = \zeta_{122}+2\zeta_{212}+\zeta_{311}+3\zeta_{14}-3\gamma\theta_2+\mu\theta_3 -\zeta_{22} -2\zeta_{13} +7\mu\theta_2^2-2\mu\theta_3-3\mu\hat{\theta}_4 \\
    \beta_{13} & = 8\zeta_{13} + 4\zeta_{22} -28\gamma\theta_2+24\mu\theta_2^2-8\mu\theta_3 \\
    \beta_{14} & = -8\gamma\theta_2+10\mu\theta_2^2-4\mu\theta_3 \\
    \theta_j & = \Hm{-j}/\Hm{-1} \\
    \hat{\theta}_j & = \Hm{-j}/\Hm{-1}^2 \\
    \zeta_{13} & = \mean{\fluctint{H^{-1}}\fluctint{H^{-3}}}/\Hm{-1}^2 \\
    \zeta_{14} & = \mean{\fluctint{H^{-1}}\fluctint{H^{-4}}}/\Hm{-1}^3 \\
    \zeta_{22} & = \mean{\fluctint{H^{-2}}^2}/\Hm{-1}^2 \\
    \zeta_{212} & = \mean{H^{-2}\fluctint{H^{-1}}\fluctint{H^{-2}}}/\Hm{-1}^3 \\
    \zeta_{122} & = \mean{H^{-1}\fluctint{H^{-2}}^2}/\Hm{-1}^3 \\
    \zeta_{311} & = \mean{H^{-3}\fluctint{H^{-1}}^2}/\Hm{-1}^3.
\end{align*}
Most of the coefficients of \eqref{o4avg} are already given above; those that differ are 
\begin{align*}
    \hat{\alpha}_8 & = -4\gamma + 10\mu \theta_2 \\ 
    \hat{\alpha}_9 & = 8\mu \theta_2\Hm{-1}^{-1} \\ 
    \hat{\alpha}_{10} & = (3\mu \theta_2 - 2\gamma) \Hm{-1}^{-1} \\ 
    \hat{\alpha}_{11} & = 4 \mu \theta_2 \\ 
\end{align*}
Finally, the additional coefficients appearing in \eqref{o4avg-ttt} are
\begin{align*}
    \hat{\alpha}_4 & = \frac{4\Hm{-5} - 2\Hm{-2}\Hm{-3} }{\Hm{-1}^2} \\
    \hat{\alpha}_6 & =  \frac{ 5 \Hm{-2}\Hm{-3} - 3 \Hm{-1}\Hm{-4} - 6\Hm{-5} }{\Hm{-1}^2}.
\end{align*}

\begin{rem}\label{rem:alpha2}
It can be proved that $\alpha_3\leq 0$, $\alpha_1 < 0$, and $\alpha_2 < 0$, for arbitrary positive profiles $H$. In particular, 
 equalities hold only if the function $H$ is constant.
We can start by proving that   $\alpha_3 \leq 0$.
From Eq.~\eqref{eq:alpha3}, the numerator $\mathcal{N}_3$ of $\alpha_3$ is given by
\[
    \mathcal{N}_3 = \mean{H^{-2}}^2 -  \mean{H^{-1}} \ \mean{H^{-3}}.
\]
while the denominator is positive. 
The sign of $\mathcal{N}_3$ is a consequence of the following theorem of real analysis \cite[Proposition 6.1]{folland1999real}.
\end{rem}


\begin{prop}\label{prop:folland}
    If $0 < p < q < r \leq \infty$,  then $L^p\cap L^r \subset L^q$ and 
    $\|f\|_q \leq \|f\|_p^{\lambda} \|f\|_r^{1-\lambda}$, where 
    \[
        \lambda = \frac{q^{-1}-r^{-1}}{p^{-1}-r^{-1}}.
    \]
\end{prop}

By choosing $p=1$, $q=2$, and $r=3$ in the above proposition, we get that
$\|f\|^4_2\leq \|f\|_1 \|f\|_3^3$, which implies $\mean{H^{-2}}^2 - \mean{H^{-3}} \ \mean{H^{-1}} \leq0$. 

Now from Eqs.~\eqref{eq:alpha1} and \eqref{eq:alpha3} it follows $\alpha_1 = 2\alpha_3 \mean{H^{-1}}-\mean{H^{-3}}/\mean{H^{-1}}$, therefore $\alpha_1<0$.

Finally, the denominator  of $\alpha_2$ in Eq.~\eqref{eq:alpha2} 
is positive and its numerator $\mathcal{N}_2$ can be written as 

\[
    \mathcal{N}_2 = 3 \left( \mean{H^{-2}}^2 - \mean{H^{-4}} \right) - 2  \mean{H^{-1}} \ \mean{H^{-3}}.
\]
The sign of the first term follows from the observations that 
\[
    \int_0^1(f(y)-\mean{f})^2\, dy \geq 0, \> 
    {\rm with} \> f= H^{-2}, 
\]
and that the second term is negative. 

\subsection{Piecewise-constant bathymetry}\label{sec:pwc-coeff}
We can compute explicit expressions for the coefficients for particular forms of $H$.
Here we just consider a simple piecewise-constant bathymetry
\begin{align}
    H(y) = \begin{cases} \frac{1}{d_1} & 0 < y \le 1/2 \\
                        \frac{1}{d_2} & 1/2 < y \le 1.
            \end{cases}
\end{align}
In this case we find that
\begin{align*}
    \mu & = \frac{(d_1-d_2)^2}{48 (d_1+d_2)^2} \\
    \nu_1 & = \mu/40 \\
    \nu_2 & = 3 \mu/40 \\
    \nu_1 + \nu_2 - \mu^2 & = \frac{(d_1-d_2)^2 \left(19 d_1^2+58 d_1
   d_2+19 d_2^2\right)}{11520 (d_1+d_2)^4}.
\end{align*}
Note that, since $d_1$ and $d_2$ are necessarily positive, the last expression above is also positive.
With more work, this can be proven for more general piecewise-constant bathymetry 
(e.g., if the  spatial fraction for each constant part differs from 1/2).


\section{Wasserstein distance}
\label{sec:Wasserstein}
When comparing the solution of the finite volume scheme applied to the shallow water system with bathymetry and the solution of the homogeneized equations, we see that there are two types of errors:  a {\em shift\/} and a  {\em shape\/} error. Both of them will be taken into account altogether if we use the Wasserstein metric when comparing the two measures. 
Let us assume $\eta(x,t)-\eta_0$ is non-negative, which can always be obtained by a suitable choice of the constant $\eta_0$.

Given that mass is conserved both in the solution of the shallow water system and in the solution of the homogeneous equations, we can normalize the quantity
$\eta(x,t)-\eta_0$ by its integral, so that 
$f(x,t) = (\eta(x,t)-\eta_0)/\int_0^L (\eta(\tilde{x},t)-\eta_0)\,d\tilde{x}$ is the density function of a measure. Here $L$ is sufficiently large that no signal reaches it up to final time $T$.

Let us call $\mu_1(t)$ and $\mu_2(t)$ two time dependent measures with densities, respectively $f_1(x,t)$ and $f_2(x,t)$, the first referring for example to the detailed numerical solution of the shallow water system, and the second the one obtained by one of the homogenized equations. Then the Wasserstein distance in $L^1$ can be computed as 
\[
W_1(\mu_1,\mu_2) := \int_\mathbb{R} |F_1(x,t)-F_2(x,t)| \,dx
\]
(see, e.g., \url{https://en.wikipedia.org/wiki/Wasserstein_metric}),
where $F_1$ and $F_2$ denote the distribution functions of the two measures $\mu_1$ and $\mu_2$, i.e.\ the primitives of $f_1$ and $f_2$.
In the case of a distribution with compact support density which translates rigidly, say if $\mu_1$ with density $f_0(x)$ and $\mu_2$ with density $f_0(x-ct)$, then  this quantity computes the space shift between the two distributions, i.e.\ $ct$. 

%

Note that even if $\eta_0$ can be chosen to make $\eta(x,t)-\eta_0$ non-negative, we prefer to keep $\eta_0$
to be the unperturbed profile, and, with a slight abuse, we apply the above formulas to such function, so that the moving wave train is very close to a function with compact support.

\section*{Acknowledgments}
The authors would like to thank Prof.~Giuseppe Di Fazio who provided the reference for the proof of Proposition \ref{prop:folland}. 
 G.~Russo would like to thank the Italian Ministry of University and Research (MUR) to support this research with funds coming from PRIN Project 2017 (No. 2017KKJP4X entitled “Innovative numerical methods for evolutionary partial differential equations and applications”), PRIN Project 2022 (No. 2022KA3JBA entitled  ``Advanced numerical methods for time dependent parametric partial differential equations with applications'') and KAUST for hosting him during the time this work was completed.

\newpage

\bibliographystyle{siamplain}
\bibliography{refs}

\end{document}